\documentclass[11pt]{amsart}
\usepackage{graphicx, amssymb}

\usepackage{enumerate}
\numberwithin{equation}{section}

\newtheorem{Thm}{Theorem}[section]
\newtheorem{Rmk}[Thm]{Remark}
\newtheorem{Prop}[Thm]{Proposition}

\newtheorem{Lem}[Thm]{Lemma}
\newtheorem{Lemma}{Lemma}
\newtheorem{Def}[Thm]{Definition}
\def\eqdefa{\buildrel\hbox{\footnotesize def}\over =}

\def\bC {\mathbf{C}}

\def\bR {\mathbf{R}}

\def\bT {\mathbf{T}}
\def\bZ {\mathbf{Z}}

\def\bP {\mathbb{P}}

\def\cB {\mathcal{B}}

\def\cF {\mathcal{F}}

\def\eps {{\epsilon}}

\def\om {{\omega}}

\def\la {\langle}
\def\ra {\rangle}
\def \La {\bigg\langle}
\def \Ra {\bigg\rangle}

\def\d {{\partial}}

\newcommand{\sgn}{\operatorname{sign}}

\newcommand{\ba}{\begin{aligned}}
\newcommand{\ea}{\end{aligned}}

\newcommand{\be}{\begin{equation}}
\newcommand{\ee}{\end{equation}}
\newcommand{\bu}{\bar u}

\let\ds=\displaystyle

\begin{document}
\title[Rotating fluids with resonant surface stress ]
      {Mathematical study of  rotating fluids \\with resonant surface stress}

\author{Anne-Laure Dalibard  \and Laure Saint-Raymond}

\date{\today}

\maketitle

\begin{abstract}
 
We are interested here  in describing the linear response of a highly rotating fluid to some surface stress tensor, which admits fast time oscillations and
may be resonant with the Coriolis force.  In addition to the usual
Ekman layer, we exhibit another - much larger -  boundary layer, and
we prove that for large times, the effect of the surface stress may no longer be localized in the vicinity of the surface. 
 From a mathematical point of view,  the main novelty here  is to
introduce some systematic approach for the study of boundary effects.

\end{abstract}

The goal of this paper is to understand the influence of a surface stress  - depending on time - on the evolution of an incompressible and homogeneous rotating fluid. More precisely, we are interested in the effects of a resonant forcing, i.e. of a stress oscillating with the same period as the rotation of the fluid. 

In the non-resonant case, the works by Desjardins and Grenier \cite{DG} then by Masmoudi \cite {M} show that the wind forcing creates essentially some boundary layer in the vicinity of the surface, which contributes to the mean motion by a source term, known as the Ekman pumping. For a precise description of the method leading to such convergence results, we refer to the book \cite{CDGG} by Chemin, Desjardins, Gallagher and Grenier.

Here the situation is much more complicated since the resonant part of the forcing will be proved to generate another boundary layer with a different typical size, and may overall destabilize the whole fluid with the apparition of a vertical profile. We give here a precise description of these (linear) effects of the Coriolis force in presence of resonant wind.

\section{Introduction}

Let us first present the mathematical framework of our study.

\subsection{A linear model for rotating fluids} $ $

$\bullet$ Our starting point is the linear version of the {\bf homogeneous incompressible Navier-Stokes system in a rotating frame}
\begin{equation}
\label{NS0}
\begin{aligned}
\d_t u +\nabla p =\cF +  u\wedge\Omega \,,\\
\nabla \cdot u =0\,,
\end{aligned}
\end{equation}
where $\cF$  denotes the  frictional
force acting on the fluid, $\Omega$ is the rotation vector, and $p$ is the pressure defined as the Lagrange multiplier associated with the incompressibility constraint.
We assume that equation \eqref{NS0} is already in a nondimensional form, meaning that all unknowns and parameters are dimensionless. For a precise dimensional analysis, we refer for instance to \cite{LTW} (section I.3).

We assume further that the rotation vector $\Omega$ is constant, homogeneous, and has constant vertical direction, which we denote by $e_3.$ Moreover, we wish to study the limit of fast rotation, i.e. $|\Omega|\to \infty.$ Hence, we set
$$
\Omega:= \frac{1}{\eps} e_3,\quad\text{with }\eps\to 0,
$$
where the parameter $\eps$ is called the Rossby number. 

$\bullet$ We consider the motion {\bf in some horizontal strip}
$$ \omega = \omega_h \times [0,1]$$
where the bottom and upper surface of the fluid are assumed to be flat at $z=0$ and $z=1$. For the sake of simplicity, we restrict our attention to the case when  $\omega_h=\bT^2$ is the the two-dimensional torus.

As boundary conditions on the upper surface, we enforce
\begin{equation}
\label{top}
\begin{aligned}
u_{3|z=1}=0,\\
\d_z u_{h|z=1}=\beta\sigma^\eps\,,
\end{aligned}
\end{equation}
where $\beta$ is a positive constant and $\sigma^\eps$ is a given stress tensor of order one, describing  the { stress on the surface of the fluid}. 

At the bottom we use the Dirichlet boundary condition
\begin{equation}
\label{bottom}
u_{|z=0}=0.
\end{equation}

$\bullet$
At last, we assume that {\bf frictional forces} $\cF$ are given by
$$
\cF=\Delta_h u + \nu \d_{zz} u,
$$
 such a choice is classical in the rotating fluids literature, see for instance \cite{CDGG,M,MR}. We refer to paragraph \ref{oceanic} for an attempt of justification in a geophysical context.

Hence, our goal is to study the asymptotic behaviour as $\eps\to 0$ of the solution of
\begin{equation}
\label{NS1}
\begin{aligned}
\d_t u + \frac{1}{\eps } e_3\wedge u +\nabla p -  \Delta_h u - \nu \d_{zz} u =0 \,,\\
\nabla \cdot u =0\,,
\end{aligned}
\end{equation}
supplemented with the boundary conditions \eqref{top}-\eqref{bottom}, depending of the order of magnitude of the vertical viscosity $\nu$.

\medskip
\subsection{Formal study of the asymptotics} $ $

The system \eqref{top}-\eqref{NS1} has already been studied by several authors, see for instance \cite{M,DG}, and also \cite{CDGG,MR} when Dirichlet boundary conditions are enforced at the top and at the bottom. Before describing the precise issues we wish to study in the present paper, let us recall briefly some of the main results and techniques for singular perturbation problems.

$\bullet$ The first step is to determine the {\bf geostrophic motion}. The only way to control the Coriolis force as $\eps \to 0$ is to balance it with the pressure gradient term (see for instance \cite{LTW}). Hence in the limit, $e_3\wedge u$ must be a gradient
\be\label{eq:geostrophic} e_3\wedge \bar u^{int}_{mean}=-\nabla p\ee
which leads to
$$ u^{int}_{mean} =\nabla_h^\perp p$$
where the limit pressure and thus the limit velocity  are independent of $z$.
In particular, $ u^{int}_{mean}$ is a two-dimensional, horizontal, divergence-free vector-field. The fluid being limited by rigid boundaries, from above and below, the divergence-free condition leads indeed to $u_3=0$ (at least to first order in~$\eps$). 
In other words, all the particles which have the same $x_h$ have the same velocity. The particles of fluid move in vertical columns, called  Taylor-Proudman columns. That is the main effect of rotation and a very strong constraint on the fluid motion.

As the domain evolution is limited by two parallel planes, the height of Taylor-Proudman columns is constant as time evolves, which is compatible with the incompressibility constraint. We can then prove that the columns move freely and in the limit of high rotation the fluid behaves like a two-dimensional incompressible fluid. Integrating the motion equation (\ref{NS1}) with respect to $z$ and taking formal limits as $\eps \to 0$ leads indeed  to
\begin{equation}
\label{formal-mean}
\begin{aligned}
\d_t  u^{int}_{mean}+\nabla_h p = \Delta_h  u^{int}_{mean},\\
\nabla_h \cdot  u^{int}_{mean} =0\,.
\end{aligned}
\end{equation}

Note however that on the boundary of the domain, where the velocity is prescribed, the $z$ independence is violated. That leads to {\bf vertical boundary layers}  modifying the limit equation (\ref{formal-mean}),  which will be investigated in the rest of the paper.

\medskip
$\bullet$ Before starting with the precise study of these boundary layers, let us now describe what happens for the three-dimensional {\bf ageostrophic part} of the initial data, i.e. the part of the initial data that does not satisfy the geostrophic constraint \eqref{eq:geostrophic}. 
The dominant process is then governed by the  Coriolis operator
\be L:u\in V_0\mapsto \bP(  e_3\wedge u )\in V_0 ,\label{Coriolis}\ee
where  $V_0$ denotes the subspace of $L^2(\omega)$ of divergence-free vector fields having zero flux both through the bottom and through the surface
$$V_0=\{ u\in L^2([0,1]\times \bT^2)\quad /\quad  \nabla\cdot u=0\hbox{ and } u_{3|z=0}=u_{3|z=1} =0\},$$
and $\bP$ denotes the orthogonal projection onto $V_0$ in $L^2(\omega)$. Notice that in general, $V_0$ is strictly smaller than the space of divergence-free vector fields in $L^2(\om)$, and consequently 
$\bP$ is different from the Leray projector.

The equation
$$\eps \d_t u+ Lu=0$$
 turns out to describe the propagation of waves, called  Poincar\'e waves. More precisely, 
one can prove (see for instance \cite{LTW}, \cite{CDGG} and Appendix A at the end of this paper for more details) that there exists a hilbertian basis of $V_0$, denoted by $(N_k)_{k\in\bZ^3\setminus\{0\}}$, constituted of eigenvectors of the linear penalization: for all~$k\in\bZ^3\setminus\{0\}$, we have
\be\label{basis}
LN_k=\bP(e_3\wedge N_k)=i\lambda_kN_k,\quad\text{where }\lambda_k = -\frac{k_3\pi}{\sqrt{|k_h|^2 + (\pi k_3)^2}}.
\ee

That means that the three-dimensional part of the initial data generates waves, which propagate very rapidly in the domain (with a speed of order $\eps^{-1}$).
The time average of these waves vanish, like their weak limit, but they carry a non-zero energy. 

\medskip
\subsection{Resonant forcing}

In view of the remarks of the previous paragraph, it seems interesting, in order to study possible resonances between the surface stress and the Coriolis operator $L$, to consider  in \eqref{top}  a stress tensor of the form
$$
\sigma^\eps(t,x_h)=\sigma\left(\frac{t}{\eps}, x_h \right),
$$
with $\sigma\in L^\infty([0,\infty) \times \bT^2)$ almost periodic in its first variable, i.e.
\be\label{hyp:sigma}
\sigma(\tau,x_h)=\sum_{k_h\in \bZ^2}\sum_{\mu \in M} \hat \sigma(\mu,k_h) e^{i\mu \tau} e^{ik_h\cdot x_h},
\ee
where $M$ is a finite set. The corresponding boundary layer terms are then expected to oscillate with the frequencies $\mu/\eps$, with either $\mu\in M$ or $\mu=-\lambda_k$ for some $k\in\bZ^3.$ The construction of such boundary layer terms is relatively well understood (see for instance \cite{CDGG,M,MR}), insofar as $\mu\neq \pm 1$. When $|\mu|=1$, the classical construction of boundary layers fails; the usual way to get round this difficulty is to assume that the initial data and the stress tensor satisfy some spectral assumptions, in order to avoid the apparition of the frequencies $\mu=\pm 1$ altogether.

Our goal in this paper is precisely to study the {\bf influence of such resonant frequencies }on the global behaviour of the fluid, starting with the boundary layers. To that end, we have developed a systematic way of computing the boundary layer profiles associated with some given boundary conditions; our main result in that regard is stated in the next paragraph, and proved in section \ref{boundary-decomposition}. Next, we use the boundary layer profiles so defined in order to construct an approximate solution for equation \eqref{NS1}, supplemented with \eqref{top}-\eqref{bottom}, and we prove a strong convergence result for \eqref{NS1}.

\section{Main results}

\label{results-par}

\subsection{Description of the boundary layers} $ $

We begin with the construction of boundary layers. Let us first emphasize that since equation \eqref{NS1} is linear, we can work with a finite number of Fourier modes in the horizontal domain  and in  time. Note that on the contrary, because of the boundary conditions at $z=0$ and $z=1$, there is a strong coupling between the vertical modes.

Hence, let $N>0$ be an arbitrary integer, and let $M_0$ be a finite set such that $M\subset M_0$. We consider some arbitrary boundary conditions $\delta^0_h$ and $\delta^1_h$ which take the form
\be
\delta^j_h(\tau,x_h)= \sum_{|k_h|\leq N} \sum_{\mu\in M_0}\hat\delta^j_h(\mu,k_h) e^{i\mu \tau} e^{ik_h\cdot x_h},\quad j=0\text{ or }1.\label{hyp:deltajh}
\ee
Here and in the whole paper, the superscript $0$  (resp. $1$) stands for functions associated with some boundary conditions at the bottom (resp. at  the surface).

Our goal is to construct some stationnary boundary layer profiles, denoted by $v^0, v^1,$ which have respectively exponential decay with respect to $z$ and $1-z$, are exact solutions of equation \eqref{NS1}, and satisfy 
\be
\begin{aligned}
v^0_{h|z=0}(t,x_h)=\delta^0_h\left( \frac{t}{\eps} , x_h\right),\\ 
\d_zv^1_{h|z=0}(t,x_h)= \delta^1_h\left( \frac{t}{\eps} , x_h\right).
\end{aligned}
\label{CL:vj}\ee
Notice that we do not enforce boundary conditions on both sides for $v^j_h$, and that we do not specify the boundary condition on the vertical component of each function $v^j$: indeed, the vertical component of $v^j$ is dictated by the assumption that $v^j$ is divergence free and that its dependance on the vertical variable $z$ is given by a decaying exponential. Similarly, the trace of $v^0$ at $z=1$ is imposed by the exponential profile condition.
At last, we do not specify any initial data for $v^0,v^1$, for the same reasons as above; we only require that $\| v^j_{|t=0}\|_{L^2}=o(1)$ as $\eps,\nu\to 0.$

However that construction fails if some particular coefficients $\hat \delta^j(k_h,\mu)$ in the boundary condition are not identically zero (see Remark \label{resonant-rmk} on page 15).
This leads to the following definition:

\begin{Def}\label{def:nonres}Assume that the boundary conditions $\delta^j_h$ are given by
$$
\delta^j_h(t,x_h)=\sum_{k_h}\sum_{\mu}\hat\delta^j_h(\mu,k_h) e^{ik_h\cdot x_h} e^{i\mu\frac{t}{\eps}}.
$$

 We define the {\bf resonant part $\delta^j_{h,res}$ } of the boundary conditions by
$$
\delta^j_{h,res}:={1\over {2}}  \La \begin{pmatrix} 1 \\i \\ 0 \end
{pmatrix}\big| \hat\delta^j_h(1,0) \Ra  \begin{pmatrix} 1 \\i \\ 0 \end{pmatrix}
e^{i\frac t\eps}+{1\over {2}}  \La \begin{pmatrix} 1 \\-i \\ 0 \end
{pmatrix}\big| \hat\delta^j_h(-1,0) \Ra  \begin{pmatrix} 1 \\-i \\ 0 \end
{pmatrix} e^{-i\frac t\eps}$$
 We will say that a boundary condition $\delta^j_h$ is {\bf non-resonant} if $\delta^j_{h,res}=0$.

\end{Def}

In the resonant case, we will indeed see that the boundary profiles are not stationnary. More precisely, we will prove the following result

\begin{Thm}
 Let $\delta^0_h, \delta^1_h$ be given by \eqref{hyp:deltajh}. Then there exist $v^0, v^1$ which are exact solutions of \eqref{NS1} supplemented with \eqref{CL:vj}, and such that $v^0$ decays exponentially with $z$, and $v^1$ with $1-z$. Moreover, each function $v^j$ ($j=0$ or $1$) can be written as 
$$
v^j= \bar v^j + \tilde v^j + v^j_{res}
$$
where the stationnary boundary profiles $\bar v^j$, $\tilde v^j$  satisfy the following estimates 
\begin{equation}
\begin{aligned}
\| \bar v^j_h \|_{L^\infty(\bR^+,L^2(\omega))}+ \frac{1}{\sqrt{\eps\nu}}\| \bar v^j_3 \|_{L^\infty(\bR^+,L^2(\omega))}\leq C(\eps \nu)^{\frac{1+ 2j}{4}} \|\delta^j_h\| ,\\
\| \tilde v^j_h \|_{L^\infty(\bR^+,L^2(\omega))} + \left(\frac{\eps + \sqrt{\eps\nu}}{\eps\nu}\right)^\frac12 
\| \tilde v^j_3 \|_{L^\infty(\bR^+,L^2(\omega))}\leq C \left(\frac{\eps\nu}{\eps + \sqrt{\eps\nu}}\right)^\frac{1+2j}{4} \| \delta^j_h\|,\\
\end{aligned}\label{B-cont}
\end{equation}
while the resonant part $v^j_{res}$ satisfies
\begin{equation}
\label{B-res-cont}
\forall t\geq 0,\quad\| v^j_{res,h}(t)\|_{L^2(\omega)} \leq C (\nu t)^\frac{1+ 2j}{4} \| \delta_{h, res}^j\|,\quad v^j_{res,3}\equiv 0\,,
\end{equation}
where 
$$\|\delta_h^j\| = \sum_{\mu\in M_0} \sum _{|k_h|\leq N}  |\hat\delta^j_h(\mu,k_h)|^2$$
and $C$ is a nonnegative constant depending on $N$.

\label{BL-thm}

\end{Thm}
Theorem \ref{BL-thm} will be proved in section \ref{boundary-decomposition}. The definition of the boundary layer operator $\mathcal B$ is then as follows:
\begin{Def} Let $\delta^0_h, \delta^1_h$ be given by \eqref{hyp:deltajh}. 
We denote by $\mathcal B$ the bilinear operator such that with the notations of Theorem \ref{BL-thm},
$$
\begin{aligned}
v^0=\mathcal B(\delta_h^0, 0),\\
v^1 =\mathcal B(0,\delta_h^1).
\end{aligned}
$$

\label{BL-def}
\end{Def}

\begin{Rmk}\label{BL-rmk}
{\bf (i)} As we shall see in the course of the proof, the terms $\bar v^j$ correspond to the {\bf usual Ekman layers}, for which the typical size of the boundary layer is $\sqrt{\eps\nu}.$ The corresponding boundary conditions are given by
$$ \bar \delta^j_h(\tau,x_h)= \sum_{|k_h|\leq N} \sum_{|\mu|\neq 1}\hat\delta^j_h(\mu,k_h) e^{i\mu \tau} e^{ik_h\cdot x_h}.$$
On the contrary, the terms $\tilde v^j$ are due to the {\bf quasi-resonant modes}, for which $|\mu|=1$ and $k_h\neq 0;$ for these modes, the typical size of the boundary layer is much larger, of order $\sqrt{\eps\nu}/ (\sqrt{\eps} + (\eps\nu)^{1/4}).$ 
$$ \bar \delta^j_h(\tau,x_h)= \sum_{k_h\neq 0} \sum_{|\mu|= 1}\hat\delta^j_h(\mu,k_h) e^{i\mu \tau} e^{ik_h\cdot x_h}.$$

{\bf (ii)}  The last terms $v^j_{res}$ are due to {\bf resonant forcing} on the modes $|\mu|=1, k_h=0$. Notice that for these modes, the estimate is not global in time: indeed, the typical size of the boundary layer is $\sqrt{\nu t}.$ 
$$ \bar \delta^j_h(\tau,x_h)=  \sum_{|\mu|= 1}\hat\delta^j_h(\mu,0) e^{i\mu \tau}.$$
In particular, for large times ($t\gg \nu^{-1}$), the boundary layer penetrates the interior of the fluid.

{\bf (iii)} As outlined above, the boundary layer term $v^0$ (resp. $v^1$) does not vanish on $z=1$ (resp. on $z=0$). Precisely, we find that there exists a positive constant $C$ (depending on $N$ and $M$) such that
$$
\begin{aligned}
\bar v^0_{|z=1}= O \left( \exp\left( -\frac{C}{\sqrt{\eps\nu}} \right) \right), \ 
\bar v^1_{|z=0}= O \left( \sqrt{\eps\nu} \exp\left( -\frac{C}{\sqrt{\eps\nu}} \right) \right), \\
\tilde v^0_{|z=1}= O \left( \exp\left( -\frac{C}{(\eps\nu)^{1/4}} \right) \right), \ 
\tilde v^1_{|z=0}= O \left( (\eps\nu)^{1/4} \exp\left(- \frac{C}{(\eps\nu)^{1/4}} \right) \right), \\
v^0_{res|z=1}=O \left( (\nu t)^{1/2}\exp\left(- \frac{1}{{4\nu t}} \right) \right),\ 
v^1_{res|z=0}=O \left( (\nu t)^{3/2}\exp\left(- \frac{1}{{4\nu t}} \right) \right).
\end{aligned}
$$

\end{Rmk}

\subsection{Construction of approximate solutions to \eqref{NS1}-\eqref{top}\eqref{bottom}} $ $

Once the mechanism of construction of boundary layers is understood, one possible application lies in the definition of an approximate solution of equation \eqref{NS1}, with a view to derive a limit system for this equation. This approximate solution is the sum of boundary  terms $u^{BL}$, obtained as above, and interior terms $u^{int}$.

\medskip
$\bullet$ Hence, we now explain the asymptotic behaviour of the {\bf interior part of the solution}.
Following the multi-scale analysis initiated in the previous paragraph, we expect the solution $u_\eps$ to \eqref{NS1}  to  behave like some function  $\exp(-tL/\eps )u^{int}_L (t)$, where $L$ is the Coriolis operator defined by \eqref{Coriolis}.

In order to understand the evolution with respect to the slow time variable, the idea is then to get rid of the penalization term by  filtering out the oscillations in equation \eqref{NS1} (see \cite{grenier,schochet}), that is, by composing equation \eqref{NS1} by the Coriolis semi-group $\exp(tL/\eps)$. 

The filtered function $u_{\eps,L}(t):=\exp(tL/\eps)u_\eps (t)$ satisfies a linear equation with vanishing viscosity (and without any penalization term); passing to the limit in the latter yields the so-called `envelope equation'
\be
\begin{aligned}
\d_t u^{int}_L -\Delta_h u^{int}_L  + \sqrt{\frac{\nu}{\eps}} S_{Ekman}u^{int}_L =0,\\
u^{int}_{L|t=0}=\gamma,\label{envelope}
\end{aligned}
\ee
where  $S_{Ekman}:V_0\to V_0$ is a linear, positive and continuous operator resulting from the non commutation between the vertical Laplacian $\nu \Delta_z$ with boundary conditions and the Coriolis semi-group (see \cite{CDGG} and \eqref{def:Ekman_op} below for a precise definition).

\medskip
$\bullet$ The approximation of the function $u_\eps$ constructed in this paper is actually much more precise than the mere function $\exp(-tL/\eps ) u_L^{int}$. Indeed, we will need to build {\bf boundary and corrector terms}, which are all small in $L^2$ norm, and thus do not play a role in the final convergence result, but are necessary in order that equation \eqref{NS1} is approximately satisfied.

\subsection{Convergence result}

\begin{Thm}Let $\gamma\in V_0$, and let  $\sigma$ be given by \eqref{hyp:sigma}.  
Let $u_\eps \in \mathcal C(\bR^+, V_0)\cap L^2_{\text{loc}}(\bR^+, H^1(\bT^2\times [0,1]))$ be the unique solution of \eqref{NS1} supplemented with \eqref{top}-\eqref{bottom}, and  let 
$u^{int}_L\in \mathcal C(\bR^+, V_0)\cap L^2(\bR_+, H^1_h(\bT^2\times [0,1]))$ be the solution of equation \eqref{envelope}.

Assume that $\sigma$ has a finite number of  Fourier modes, i.e. $\sigma$ satisfies \eqref{hyp:deltajh}.

Then under the technical scaling assumption \eqref{hyp:scaling} on the parameters $\eps,\nu$ and $\beta$, we have, as $\eps,\nu\to 0$, 
\be\label{conv}
u_\eps (t)-\exp\left( -\frac{t}{\eps}L \right)u^{int}_L(t)\to 0
\ee
in
 $L^\infty_{\text{loc}}(\bR^+, L^2(\bT^2\times[0,1]))$.

\label{thm}
\end{Thm}

\begin{Rmk}
{\bf (i)} That result extends previous works by Masmoudi \cite{M} and Chemin, Desjardins, Gallagher and Grenier \cite{CDGG}. They have indeed studied analogous boundary problems for rotating fluids, but have used in a crucial way  a spectral assumption on the forcing modes, which ensures that the forcing is non-resonant, or in other words that the boundary layers remain stable.

 {\bf (ii)} The above theorem holds for all values of the ratio $\nu/\eps$, but the asymptotic behaviour of $\bu_L$ depends on the scaling of $\nu/\eps$.
 
  Note that, in the case when $\eps \gg\nu$, the effects of the boundary terms, even damped by the penalization, remain localized in the vicinity of the surface and thus do not contribute to the mean motion.
  
  If $\nu/\eps\to \infty$, the vertical dissipation  damped by the penalization induces a strong relaxation mechanism, so that we expect the solution to be well approximated, outside from some initial layer, by a ``stationary" solution to the wind-driven system. That initial layer should be of size $O\left( \sqrt{\eps \over \nu}\right)$ and the relaxation should be governed by the Ekman dissipation process \eqref{envelope}.

 {\bf (iii)} If the forcing $\sigma $ bears on resonant modes only, then we are able to prove  a global  result. Precisely, assume that
$$
\sigma(\tau,x_h)=\hat \sigma^+ e^{i\tau}(1,i) + \hat \sigma^- e^{-i\tau}(1,-i).
$$
Then there exists some destabilization profile $v_\nu$ solution of the heat equation (\ref{heat-eq}) such that
\be
u_\eps(t)-\left[ \exp\left( -\frac{t}{\eps}L \right)(u^{int}_L(t) + v_\nu (t)\right]\to 0\label{conv_bis}
\ee
in $L^\infty(\bR^+, L^2(\bT^2\times[0,1]))\cap L^2(\bR^+, L^2(\bT^2\times[0,1]))$. 

In particular, for large times, $$u(t)\approx\exp\left( -\frac{t}{\eps}L \right) v_\nu (t)=O(\beta).$$
Since $\beta$ may be very large (see \eqref{hyp:scaling}), there is a { destabilization of the whole fluid} inside the domain as $t\to\infty.$
Note that the two convergences \eqref{conv} and \eqref{conv_bis} are compatible, since with assumption \eqref{hyp:scaling}, $$v_\nu=O(\nu ^{3/4}\beta)=o(1)\hbox{ in }L^2([0,T]\times \bT^2\times [0,1])$$
 for any finite time $T>0$.

\label{thm-rmk}

\end{Rmk}

\subsection{Method of proof}$ $

 Let us now give some details about our method of proof.
As the evolution equation is linear, we will  use {some superposition principle}, meaning that we will deal separately with the forcing and with the initial condition. 

\medskip
$\bullet$ More precisely, we will consider on the one hand  the wind-driven system
\begin{equation}
\label{forcing}
\begin{aligned}
\d_t u +\frac1\eps \bP(e_3\wedge u)  \Delta_h u-\nu \d_{zz} u=0,\\
\nabla \cdot u =0,\\
u_{|t=0} =0,\\
u_{|z=0}=0,\\
u_{3|z=1}=0,\quad \d_z u_{h|z=1}=\beta\sigma^\eps.
\end{aligned}
\end{equation}
For that system, we will construct an approximate solution constituted of a boundary term $u^{BL,1}$ localized near the surface, and some interior term $v^{int,1}$, which accounts for the fact that the vertical component of $u^{BL,1}$ does not match the no-flux boundary condition at the surface (see  Remark \ref{BL-rmk} \textbf{(ii)}). 

The convergence of the modes such that $|\mu|\neq 1$ is then proved using a somewhat soft argument, which can be applied with a crude approximation. 

Concerning the quasi-resonant modes, for which $|\mu|=1$ and $k_h\neq 0$, the situation is more complicated, and we have to build several correctors before reaching the adequate order of approximation.

\medskip
$\bullet$
On the other hand, we will study the initial value problem
\begin{equation}
\label{initial}
\begin{aligned}
\d_t u +\frac1\eps \bP(e_3\wedge u) -\nu_h \Delta_h u-\nu \d_{zz} u=0,\\
\nabla \cdot u =0,\\
u_{|t=0} =\gamma,\\
u_{|z=0}=0,\\
u_{3|z=1}=0,\quad \d_z u_{h|z=1}=0.
\end{aligned}
\end{equation}
Here we will use, following \cite{CDGG}, an energy method which requires to obtain a very precise approximation. A quantitative result about the required precision is given in the stopping condition in the Appendix (Lemma \ref{stop}): when the approximate solution satisfies the  hypotheses of Lemma \ref{stop}, we put an end to the construction of correctors and conclude thanks to an energy estimate, whence the name `stopping Lemma'. The approximate solution is actually obtained  as the sum of two interior terms $u^{int}$ that we seek in the form
$$u^{int} =\sum c_l N_l e^{-i\lambda _l \frac t\eps}$$
coming from the analysis of the linear penalization as an operator of $L^2$,  and two boundary terms $u^{BL,0}$. We emphasize that in the case $\nu=O(\eps)$, the construction of an approximate solution for system \eqref{initial} has already been dealt with by several authors (see \cite{CDGG,M}); we recall it here for the reader's convenience, and further extend it to the case when $\nu\gg \eps$.

Of course, in the nonlinear case the superposition principle does not hold anymore, and both systems (\ref{forcing}) and (\ref{initial}) will be coupled.

\bigskip
The next sections are  devoted to the proofs of Theorems \ref{BL-thm} and \ref{thm}. We start with a precise description of the boundary layer operator $\cB$ in Section \ref{boundary-decomposition}. We then build, in Section \ref{wind}, the approximation and prove the convergence for the (possibly resonant) wind-driven system (\ref{forcing}). For the sake of completeness, we finally study the system (\ref{initial}) which has already been dealt with in a number of mathematical papers. Let us recall that in both cases we need a refined approximation with many orders. We have then to iterate some process giving the successive correctors. Note however that we are not able to really obtain an asymptotic expansion leading to a more accurate approximation (in $L^2$ sense). At each step of the process the order of the resonances involved in the estimates is indeed increased, so that it is not possible to obtain convergent series. For  more precisions  regarding that point, we refer to the proof in Section \ref{int+BL-par}.

\section{The boundary layer operator}\label{boundary-decomposition}

This section is devoted to the proof of Theorem \ref{BL-thm}.

\subsection{Non-resonant case}$ $

We recall that the boundary conditions are given by \eqref{hyp:deltajh}.
 and that we  seek the boundary terms as a sum of oscillating modes, rapidly decaying in $z$. Our goal in this paragraph is to characterize these modes, or in other words to describe the propagation with respect to $z$ of the  boundary conditions
$$
v_{|z=0} =\delta^0_{h},\quad  \d_z v_{h|z=1} = \delta^1_{h}\,.$$

 We  will use the following Ansatz
$$ v(t,x) = v^0 (t,x) +v^1(t,x)$$
with
$$
v^j(t,x)=\sum_{\mu, k_h } V^j(\mu,k_h; x)\exp  \left(i\frac t\eps\mu\right) $$
where $\mu$ and $k_h$ are the  oscillation period and horizontal Fourier mode.

We further seek $ V^0(\mu,k_h) $ and $ V^1(\mu,k_h) $ in the form
\begin{equation}
\label{ansatz}
\begin{aligned}
V^0(\mu,k_h;x)=\hat v^0(\mu,k_h) \exp (ik_h\cdot x_h)\exp \left(-\lambda(\mu,k_h) {z\over \sqrt{\eps \nu}}\right),\\
V^1 (\mu,k_h;x)=\hat v^1(\mu,k_h) \exp (ik_h\cdot x_h)\exp \left(-\lambda (\mu,k_h){(1-z)\over \sqrt{\eps \nu}}\right)
 \end{aligned}
 \end{equation}
so that they are  expected to be localized in a neighbourhood of size $O(\sqrt{\eps \nu})$ respectively near the bottom and near the surface. Note in particular that, with such a choice, $v^0$ (resp. $v^1$) introduces only exponentially small error terms on the surface (resp. at the bottom).

Plugging this Ansatz in the system (\ref{NS1}) we get actually
 \begin{equation}
\label{linear-ansatz}
\begin{aligned}
i\mu \hat v_1-\lambda^2 \hat v_1+\eps k_h^2 \hat v_1 - \hat v_2 +\eps \nu {k_1k_2 \hat v_1 -k_1^2 \hat v_2 \over \lambda^2-\eps \nu k_h^2}=0,\\
i\mu \hat v_2-\lambda^2 \hat v_2+\eps k_h^2 \hat v_2 +\hat v_1 +\eps \nu {-k_1k_2 \hat v_2 +k_2^2 \hat v_1 \over \lambda^2-\eps \nu k_h^2}=0,\\
\sqrt{\eps \nu} (ik_1 \hat v_1 +ik_2 \hat v_2)\pm\lambda \hat v_3=0\,.
\end{aligned}
\end{equation}
which expresses the balance between  the forcing, the viscosity, the Coriolis force and the pressure.

\bigskip
 Denote by $A_\lambda$ the matrix corresponding to   (\ref{linear-ansatz}) 
 $$A_\lambda (\mu , k_h)=
 \begin{pmatrix}
\ds   i\mu -\lambda^2 +\eps k_h^2 +{\eps \nu k_1k_2 \over \lambda^2 -\eps \nu k_h^2} & \ds-1-{\eps \nu k_1^2\over \lambda^2-\eps \nu k_h^2}\\
 \ds1+{\eps \nu k_2^2\over \lambda^2-\eps \nu k_h^2} &\ds  i\mu -\lambda^2 +\eps k_h^2 -{\eps \nu k_1k_2 \over \lambda^2 -\eps \nu k_h^2}
 \end{pmatrix}.
 $$

\noindent Classical results on boundary layers are then based on the fact that $|\mu|\neq 1$, which ensures that the matrix
$$
\begin{pmatrix}
\ds  \mu   & \ds i \\
 \ds -i &\ds \mu \end{pmatrix}
 $$
 is hyperbolic in the sense of dynamical systems, i.e. that its eigenvalues have non zero real parts. In particular, there exist two complex numbers $\lambda =\lambda(\mu,k_h)$ with nonnegative real parts such that $\det A_\lambda =0$.
 
 This feature, as well as general properties of the system, is therefore stable by small perturbation. The method consists then in neglecting the perturbation, i.e. the pressure and horizontal viscosity terms and to compute a solution to
 $$\d_t v +e_3\wedge v-\nu \d_{zz }v =0$$
 with suitable boundary conditions.
 
 Now, if $|\mu|=1$, the matrix 
$$
\begin{pmatrix}
\ds  \mu   & \ds i \\
 \ds -i &\ds \mu \end{pmatrix}
 $$
 admits zero as an eigenvalue, and we expect its behaviour to be very sensitive to perturbations. Actually we will distinguish between two cases
 
 \begin{itemize}
 
 \item either $k_h\neq (0,0)$ and we will prove that the same type of behaviour as previously occurs, with the difference that the decay rate $\lambda$ of the singular component is anomalously small. We will thus develop a general method, which can be used independently of the size of $\lambda$ (the classical method fails since the error depends on $1/\lambda^2$).
 
 \item or $k_h=(0,0)$ and we have a bifurcation. The solution $v$ is not localized anymore.
 
  \end{itemize}

 \bigskip\noindent
 {\bf  Case when $k_h\neq (0,0)$.}
 
 $\bullet$ Let us first introduce some notations in order to define an abstract framework to deal with. For the sake of simplicity, we omit here all the parameters $\mu$ and $k_h$.

 Let  $\lambda$ be such that $\det (A_\lambda) =0$, then there exists $w_\lambda$ such that
\be A_\lambda w_\lambda =0\,.\label{def:wlambda}\ee
 In other words the vector fields $W_\lambda^0$  and $W_\lambda^1$ defined by
 \begin{equation}
 \label{vlambda-def}
 \begin{aligned}
 W_\lambda^0 (t,x) = \begin{pmatrix}\ds  w_\lambda  \\ \ds { \sqrt{\eps \nu} \over \lambda} i k_h\cdot w_\lambda\end{pmatrix} \exp (ik_h\cdot x_h) \exp (i\mu\frac t\eps)\exp \left(-\lambda {z\over \sqrt{\eps \nu}}\right)\\
  W_\lambda^1 (t,x) = \begin{pmatrix}\ds { \sqrt{\eps \nu} \over \lambda} w_\lambda  \\ \ds -{\eps \nu\over \lambda^2} i k_h\cdot w_\lambda\end{pmatrix} \exp (ik_h\cdot x_h) \exp (i\mu\frac t\eps)\exp \left(-\lambda {(1-z)\over \sqrt{\eps \nu}}\right)
  \end{aligned}
 \end{equation}
 are  exact solutions to (\ref{NS1}) satisfying respectively the horizontal boundary condition
 $$
 \begin{aligned}
 W^0_{\lambda,h|z=0}= w_\lambda \exp (ik_h\cdot x_h) \exp (i\mu\frac t\eps),\\\d_z W^0_{\lambda, h|z=1}= -{\lambda \over  \sqrt{\eps \nu} } w_\lambda \exp (ik_h\cdot x_h) \exp (i\mu\frac t\eps)\exp\left(-{\lambda\over \sqrt{\eps \nu}}\right)\,,
 \end{aligned}
 $$
and
$$\begin{aligned}
\d_z W^1_{\lambda, h|z=1}= w_\lambda \exp (ik_h\cdot x_h) \exp (i\mu\frac t\eps),\\  W^1_{\lambda,h|z=0}={ \sqrt{\eps \nu} \over \lambda}w_\lambda \exp (ik_h\cdot x_h) \exp (i\mu\frac t\eps)\exp\left(-{\lambda\over \sqrt{\eps \nu}}\right)\,.
 \end{aligned}
 $$
 We have moreover the following estimates (provided that $\frac{\lambda}{\sqrt{\eps\nu}}\gg 1$)
 \begin{equation}
 \label{vl-est}
 \begin{aligned}
  W_\lambda^0 =O(1)_{L^\infty(\bR^+, L^\infty(\Omega))},\quad W_\lambda^0 =O\left( \left( {\eps \nu \over \lambda^2}\right)^{1/4} \right)_{L^\infty(\bR^+, L^2(\Omega))},\\
 W_\lambda^1 =O\left( \left( {\eps \nu \over \lambda^2}\right)^{1/2} \right)_{L^\infty(\bR^+, L^\infty(\Omega))},\quad W_\lambda^1 =O\left( \left( {\eps \nu \over \lambda^2}\right)^{3/4} \right)_{L^\infty(\bR^+, L^2(\Omega))}.
\end{aligned}
\end{equation}

 \medskip
We  intend to build one  particular solution to  (\ref{NS1}) satisfying the horizontal boundary condition
 $$
 \begin{aligned}
 v_{ h|z=0}= \delta_h^0,\\
 \d_z v_{ h|z=1}=\delta_h^1. 
 \end{aligned}
 $$
Hence, we only have to find (for all $\mu$ and $k_h$) some $w_{\lambda^-}$ and $w_{\lambda^+}$ constituting a basis of $\bC^2$.

 \medskip
$\bullet$ In order to determine some suitable $w_{\lambda^-}$ and $w_{\lambda^+}$, we have to get some asymptotic expansions of the eigenvalues and eigenvectors of $A_\lambda(\mu,k_h)$.

In view of  the previous paragraph, at leading order, we have
$$A_\lambda = \begin{pmatrix}
\ds i\mu -\lambda^2  & \ds-1\\
 \ds1&\ds i\mu -\lambda^2 
 \end{pmatrix}+o(1)$$
 so that 
 $$\det (A_\lambda) =(i\mu-\lambda^2)^2+1+o(1)=0$$
 for $(\lambda^-) ^2=i(\mu+1)+o(1)$ or $(\lambda^+)^2=i(\mu-1)+o(1)$. We further have
$$w_{\lambda^- }= (1,-i)+o(1)\hbox{ and } w_{\lambda^+} = (1,i)+o(1)$$

\smallskip
\noindent
\underline{
For $|\mu|\neq 1$}, we choose $\lambda^-$ and $\lambda^+$ to be the roots of $\det(A_\lambda)=0$ with nonnegative real parts. The previous asymptotic equivalences are then enough to prove that $$\det (w_{\lambda^-}, w_{\lambda^+}) = 2i +o(1)$$
from which we deduce that $(w_{\lambda^-}, w_{\lambda^+})$ is a (quasi-orthogonal) basis of $\bC^2$, and that we have uniform bounds (with respect to $\eps$  sufficiently small and $\nu$ bounded) on the  transition matrix $P$ and its inverse.

\smallskip
\noindent
\underbar{
For $\mu=1$} we expect   $\lambda^-$ to be given by  $(\lambda^-)^2=2i +\eta^-$ with $\eta_-=o(1)$, and $\lambda^+$ to be given by $(\lambda^+)^2=\eta^+$ with $\eta^+ =o(1)$
$$\begin{aligned}
\det (A_\lambda) =&\left(  i\mu -\lambda^2 +\eps k_h^2 +{\eps \nu k_1k_2 \over \lambda^2 -\eps \nu k_h^2} \right) \left( i\mu -\lambda^2 +\eps k_h^2 -{\eps \nu k_1k_2 \over \lambda^2 -\eps \nu k_h^2} \right) \\
&- \left(-1-{\eps \nu k_1^2\over \lambda^2-\eps \nu k_h^2}\right) \left(1+{\eps \nu k_2^2\over \lambda^2-\eps \nu k_h^2} \right)\\
= &\left(-i-\eta^- +\eps k_h^2 +{\eps \nu k_1k_2 \over 2i}\right) \left( -i-\eta^-  +\eps k_h^2 -{\eps \nu k_1k_2 \over 2i }\right)\\
&+\left(1+{\eps \nu k_1^2\over 2i}\right) \left(1+{\eps \nu k_2^2\over 2i} \right)+o(\eps )
\end{aligned}
$$
and
$$\begin{aligned}
\det (A_\lambda) = &\left(i-\eta^+ +\eps k_h^2 +{\eps \nu k_1k_2 \over \eta^+}\right) \left( i-\eta^+  +\eps k_h^2 -{\eps \nu k_1k_2 \over\eta^+ }\right)\\
&-\left(-1-{\eps \nu k_1^2\over\eta^+}\right) \left(1+{\eps \nu k_2^2\over\eta^+} \right)+O(\eps^2\nu^2/(\eta^+)^2 )
\end{aligned}
$$
from which we deduce that
$$
\begin{aligned}
 \eta^- =\eps k_h^2 +\frac14 \eps \nu k_h^2+o(\eps)\\
\eta^+ = \eps k_h^2 +{\eps \nu k_h^2\over 2i\eta^+} + o(\sqrt{\eps \nu}) + o(\eps).
\end{aligned}
$$
We have then
$$(\lambda^-)^2 =2i+O(\eps).$$
On the other hand, a  discussion taking into account the relative sizes of $\eps$ and $\nu$ shows that
$$(\lambda^+)^2 \sim \eps k_h^2 \hbox{ if } \nu<<\eps, \quad (\lambda^+)^2 \sim \pm \frac12\sqrt{\eps \nu}|k_h|(1+i)$$
while an easy argument of homogeneity gives
$$(\lambda^+)^2 \sim C(k_h) \eps \hbox{ if } \nu\sim \eps,$$
for some constant $C(k_h)$, depending only on $k_h$.
Thus there exists a constant $C(k_h)$ such that
\be\label{est:lambda+-1}
\begin{aligned}
 C(k_h)^{-1}\leq |\lambda^-(1,k_h)| \leq C(k_h),\\
C(k_h)^{-1}(\eps + \sqrt{ \eps\nu})^{1/2}\leq |\lambda^+(1,k_h)| \leq C(k_h)(\eps + \sqrt{ \eps\nu})^{1/2}.
\end{aligned}
\ee

Plugging these expansions in the formula of $A_\lambda$ leads then to
$$\begin{aligned}
w_{\lambda^- }= (1,-i+O(\eps)),\\
w_{\lambda^+} = (1,i+O(\sqrt{\eps\nu}) +  O(\eps))
\end{aligned}
$$
In particular we have
$$\det (w_{\lambda^-}, w_{\lambda^+}) = 2i +O(\eps) +O(\sqrt{\eps \nu})$$
from which we deduce that $(w_{\lambda^-}, w_{\lambda^+})$ is a (quasi-orthogonal) basis of $\bC^2$, and that we have uniform bounds (with respect to $\eps$ and $\nu$ sufficiently small) on the  transition matrix $P$ and its inverse.

\smallskip
\noindent
\underline{
For $\mu=-1$} we have in the same way
\be\label{est:lambda+-2}
\begin{aligned}
 C(k_h)^{-1}(\eps + \sqrt{ \eps\nu})^{1/2}\leq |\lambda^-(-1,k_h)| \leq C(k_h)(\eps + \sqrt{ \eps\nu})^{1/2},\\
C(k_h)^{-1}\leq |\lambda^+(-1,k_h)| \leq C(k_h)
\end{aligned}
\ee
and
$$
w_{\lambda^- }= (1,-i+O(\sqrt{\eps \nu})+ O(\eps)),\quad 
w_{\lambda^+} = (1,i+O(\eps))
$$
from which we deduce uniform bounds (with respect to $\eps$ and $\nu$ sufficiently small) on the transition matrix $P$ and its inverse.

\medskip
 $\bullet$ We then define  $  V^0(\mu,k_h)$ and   $  V^1(\mu,k_h)$ by
\begin{equation}
\label{uBLi-def}
V^j(\mu,k_h;x) \exp\left( i\mu \frac{t}{\eps} \right)= \alpha^j_- W^j_{\lambda^-}(t,x) +\alpha_+^j W^j_{\lambda^+}(t,x)
\end{equation}
where $W^j_\lambda$ is defined in terms of $w_\lambda$ by (\ref{vlambda-def}) and the coefficients $\alpha^j_-$ and $\alpha^j_+$ are defined by
\begin{equation}
\label{alpha-def}
(\alpha^j_-,\alpha^j_+)=P^{-1} \hat \delta^j_h(\mu,k_h) .
\end{equation}

 \bigskip
 \noindent
 {\bf  Case when $k_h= (0,0)$.}

That case is strongly different since there is no term of higher order in (\ref{linear-ansatz})~:
$$A_\lambda = \begin{pmatrix}
\ds i\mu -\lambda^2  & \ds-1\\
 \ds1&\ds i\mu -\lambda^2 
 \end{pmatrix}$$
  
\medskip
\noindent
\underline{For $|\mu|\neq 1$} we use exactly the same arguments as previously and define 
$\hat v^j (\mu,0)$ by formulas  (\ref{uBLi-def})(\ref{alpha-def}). 

\begin{Rmk}\label{resonant-rmk}
When $|\mu|=1$ we cannot find a basis of eigenvectors $(w_{\lambda^-}, w_{\lambda^+})$ with $\Re (\lambda^-)>0$ and $\Re(\lambda^+)>0$. One of the eigenvalue is necessarily 0, and thus the corresponding solution has no decay in $z$. In other words we do not expect the boundary terms to be localized in the vicinity of the boundary uniformly in time.
\end{Rmk}

The assumption that the boundary condition is non resonant ensures however that there is no such contribution.

\underline{If $|\mu| =1$}   we have 
$$ \lambda^{-\mu}=2\mu i \hbox{ and } \lambda^{\mu} =0$$
with
$$ w_{\lambda^-} =(1,-i) \hbox{ and } w_{\lambda^+} =(1,i)\,.$$
If we define as previously $W^j_{\lambda^{-\mu}}$  by (\ref{vlambda-def}), and 
$\alpha^j_\pm$ by (\ref{alpha-def}), we have
$$ \alpha^j_\mu =0.$$
Setting, for $j=0$ or $j=1$,
$$
V^j(\mu,0;x)\exp \left( i\mu \frac t\eps\right)=\alpha^j_{-\mu} W^j_{\lambda^{-\mu}}(t,x) 
$$
we can check that it is an exact solution to (\ref{NS1}), which further satisfies the required horizontal boundary condition.

\subsection{Resonant case}$ $ 
 
 Let us then focus on the resonant part of the motion.
 The singular component $u_{\eps, res}$ of the
velocity is a 2D vector field (depending only on $t$ and $z$),  so
that (\ref{NS1}) can be rewritten
$$\d_t  u_{\eps, res}+\frac1\eps u_{\eps, res}\wedge e_3 -\nu \d_{zz}u_{\eps, res}
=0,$$
meaning that the pressure is constant.

\medskip
$\bullet$ Therefore the equation  can be
filtered by a  simple change of unknown~:
$$v_\nu(t) ={1\over {2}}  \La \begin{pmatrix} 1 \\i \\ 0 \end
{pmatrix}\big| u_{\eps, res} \Ra  \begin{pmatrix} 1 \\i \\ 0 \end{pmatrix}
e^{-i\frac t\eps}+{1\over {2}}  \La \begin{pmatrix} 1 \\-i \\ 0 \end
{pmatrix}\big| u_{\eps, res} \Ra  \begin{pmatrix} 1 \\-i \\ 0 \end
{pmatrix} e^{i\frac t\eps}$$
A straightforward computation   leads then  to
\begin{equation}
\label{heat-eq}
\d_t  v_\nu -\nu \d_{zz} v_\nu =0,
\end{equation}
which is nothing else than the heat equation with small conductivity $
\nu$. We therefore expect  the boundary effects to  remain localized (in $L^2$ sense) in  layers of
size $O(\sqrt{\nu t})$ near the boundaries.

\medskip
$\bullet$ Let us then introduce a boundary layer approximation $$v_{L, res}=v^0_{L, res}+v^1_{L, res}$$ for $v_\nu$.
The heat equation on $v^j_{L, res}$  is supplemented with the boundary condition
$$ v^0_{L,res|z=0}=\delta^0_{L,res},\quad \d_z v^1_{L,res|z=1} =\delta^1_{L,res}
$$
and the initial condition
$$  v^j_{L,res|t=0}=0\,.$$
 Notice that once again we do not enforce boundary conditions on both sides for $v^j_{L,res}$~:  the trace of $v^j$ at $z=1-j$ will be imposed by the exponential profile condition.
We indeed seek  $v_{L,res}^j$  in the form of   self similar profiles
\begin{equation}
\label{vres-def}
 v^0_{L, res}= \varphi ^0\left({ z \over \sqrt{\nu t}}\right),\quad  \d_z v^1_{L, res}= \varphi ^1\left({ (1-z) \over \sqrt{\nu t}}\right) .
 \end{equation}
We then get
$$-\frac12 X\varphi'(X) -\varphi''(X)=0,$$
from which we deduce that
$$\varphi'(X) =\varphi'(0) \exp \left( - \frac14 X^2\right)\,,$$
and
$$\varphi(X) =- \int_X^{+\infty}\varphi'(0) \exp \left( - \frac14 Y^2
\right)dY\,.$$
We thus choose
$$(\varphi^j)'(0) =- \delta^j_{L,res} \left(\int_0^{+\infty} \exp \left( - \frac14 Y^2
\right)dY\right)^{-1}=-{1 \over \sqrt{\pi}}\delta^j_{L,res}.$$
Note that, in order that $v^1_{L,res}$ satisfies the heat equation (\ref{heat-eq}), we have to further impose that $v^1_{L,res} (-\infty)=0$.

\medskip
$\bullet$
We  deduce that
$$\begin{aligned}
 v^0_{L, res} &= {1 \over \sqrt{\pi}}\delta^0_{L,res} \int _{ z \over \sqrt{\nu t}}^{+\infty}e^{-
\frac{Y^2}4 }dY\\
&\sim _{z\neq 0} {2 \over \sqrt{\pi}}\delta^0_{L,res} \left( { z \over \sqrt
{\nu t}}\right)^{-1} \exp \left( - \frac14  \left( { z \over \sqrt
{\nu t}}\right)^2\right)
\end{aligned}$$
Similarly, we have
$$ v^1_{L, res} (t,z) =\int_{-\infty}^z \varphi^1\left({ 1-z' \over \sqrt{\nu t}}\right)dz',$$
with
$$\varphi^1\left({ 1-z \over \sqrt{\nu t}}\right)\sim_ {z\neq 1} {2 \over \sqrt{\pi}}\delta^1_{L,res} \left( { 1-z \over \sqrt
{\nu t}}\right)^{-1} \exp \left( - \frac14  \left( { 1-z \over \sqrt
{\nu t}}\right)^2\right).$$
Therefore
$$ v^1_{L, res} (t,z)\sim_ {z\neq 1} -{4(\nu t)^{3/2} \over \sqrt{\pi}}\delta^1_{L,res}  (1-z) ^{-2} \exp \left( - \frac14  \left( { 1-z \over \sqrt
{\nu t}}\right)^2\right).$$
In particular  $ v^j_{L, res}$ is exponentially small outside from a
layer of size $O(\sqrt{\nu t})$.

\subsection{Continuity estimates}$ $

 We now turn to the derivation of the estimates of Theorem \ref{BL-thm}.
 
Thanks to the previous paragraph, the resonant part of the boundary layer, namely $v^j_{res}$ defined by \eqref{vres-def}, satisfies the third estimate in (\ref{B-cont}).

We then split $v^j-v^j_{res}$ according to the size of the boundary layers
$$\begin{aligned}
\bar v^j = \sum_{k_h}\sum _{\mu \sigma\neq 1} \alpha_\sigma^j(\mu,k_h) W^j_{\lambda^\sigma(\mu,k_h)},\\
\tilde v^j = \sum_{k_h\neq 0} \sum _{\mu \sigma=1} \alpha_\sigma^j(\mu,k_h) W^j_{\lambda^\sigma(\mu,k_h)}
\end{aligned}
$$
By definition of $\alpha _\sigma^j(\mu,k_h)$ and $W^j_{\lambda^\sigma(\mu,k_h)}$, we then obtain the estimates 
$$\| \bar v^j_h \|_{L^2}+(\eps\nu)^{-1/2} \| \bar v^j_3 \|_{L^2}\leq C\beta^j(\eps \nu)^{\frac{1+2j}{4}} \| \delta^j_h\|$$
for the classical boundary layer, and 
$$
 \| \tilde v^j_h \|_{L^2(\omega)}+\frac{\sqrt{\eps} + (\eps\nu)^{1/4}}{(\eps\nu)^{1/2}}\| \tilde v_3 \|_{L^2(\omega)} \leq C\beta^j\left(\frac{\eps\nu}{\eps + \sqrt{\eps\nu}}\right)^{\frac{1+2j}{4}}  \|\delta^j_h\|
$$
for the quasi-resonant  boundary layer.


\section{Study of the wind-driven part of the motion}\label{wind}

This section is devoted to the proof of Theorem \ref{thm} in the case
where the initial data $\gamma$ vanishes. In other words, we study
here the asymptotic behaviour of the system \eqref{forcing}. Our goal
is to prove that under a technical scaling assumption which will be
precised later on, the solution $u$ of \eqref{forcing} converges
towards zero in $L^\infty_\text{loc}(\bR_+, L^2(\om))$ as $\eps,\nu\to
0$.

As explained in section \ref{results-par}, the method of proof relies
on the construction of an approximate solution $u_{app}$, defined as
the sum of boundary layer terms obtained thanks to Theorem
\ref{BL-thm}, and interior terms which will be determined by a
filtering process. The presence of these interior terms is due to the
fact that the vertical components of the boundary layer terms
constructed in Theorem \ref{BL-thm} do not vanish on $z=0$ and $z=1$.
More importantly, the traces of these boundary layer terms do not
satisfy the assumptions of the stopping Lemma \ref{stop} in Appendix
B, which quantifies the order of approximation required for $u_{app}$.
Hence in general, the approximate solution is constituted of several
correctors, which all vanish in $L^2$ norm.

The different modes of the wind stress $\sigma$ will be treated
independently of each other. Indeed, in the case where the stress
$\sigma $ does not have any quasi-resonant mode, it will be sufficient
to construct a very crude approximation, constituted merely of one
boundary layer term and one additional corrector. On the other hand,
the vertical components of the quasi-resonant boundary layer terms
have a much larger trace on $z=1$ and $z=0$ than the classical ones,
as can be seen in inequalities \eqref{B-cont}. Consequently, the
quasi-resonant part of the stress $\sigma$ will require a much more
refined approximation, with several orders of boundary layer terms and
interior terms.

The organization of this section is as follows: first, we give in
paragraph \ref{par:stab} a general convergence result for the system
\eqref{forcing}. Then, in paragraph \ref{par:approx1}, we construct
the first orders of the approximate solution $u_{app}$. In paragraph
\ref{par:forcing-nonres}, we conclude in the case when there is no
quasi-resonant mode $|\mu|=1$, $k_h\neq 0$. At last, we prove the
theorem for the quasi-resonant part of the stress $\sigma $ in
paragraph \ref{par:forcing-res}. At each step, we give some sufficient
assumptions on the  parameter $\beta$, and at the end of the proof, we
only keep the most restrictive ones, which will lead to the scaling
assumption \eqref{hyp:scaling}.

\subsection{Some stability inequality for the wind-driven system
(\ref{forcing})}

\label{par:stab}

As mentioned in Section 2, for the non-resonant part of wind-driven
system (\ref{forcing}), we will only need a rather crude approximation
of the solution. We have indeed the following

\begin{Prop}\label{energy1}
Denote by $u_\eps$ the solution to (\ref{forcing}) and by $u_{app}$ any
approximate solution in the sense that
\begin{equation}
\label{forcing-app}
\begin{aligned}
\d_t u_{app} +\frac1\eps \bP(e_3\wedge u_{app}) -\Delta_h u_{app}-\nu
\d_{zz} u_{app}=\eta,\\
\nabla \cdot u_{app} =0,\\
u_{app|t=0} =\eta_{ini},\\
u_{app,3|z=0}=0,\quad u_{app,h|z=0}=\eps \eta_0,\\
u_{app,3|z=1}=0,\quad \d_z u_{app,h|z=1}=\beta{\sigma}^\eps+\eta_1,
\end{aligned}
\end{equation}
with $\eta\to 0$ in $L^2([0,T]\times\omega)$, $\eta_{ini}\to 0$ in
$L^2(\omega)$ and $\eta_0,\nu^{3/4}\eta_1 ,\eps \d_t \eta_0 \to 0$ in
$L^2([0,T]\times\omega_h)$.
Then as $\eps,\nu\to 0$,
$$\begin{aligned}
   \| u_\eps-u_{app}\|_{L^\infty([0,T], L^2(\om))}\to 0,\\
\| \nabla_h(u_\eps-u_{app})\|_{L^2([0,T]\times\om)}+ \sqrt{\nu}\|
\d_z(u_\eps-u_{app})\|_{L^2([0,T]\times\om)} \to 0.
  \end{aligned}
$$
\end{Prop}

\begin{proof}

$\bullet$ The first step consists in building a family $w$ such that
$v_{app} \eqdefa u_{app}+w$ satisfies
\begin{equation}
\label{forcing-app*}
\begin{aligned}
\d_t v_{app} +\frac1\eps \bP(e_3\wedge v_{app}) -\Delta_h v_{app}-\nu
\d_{zz} v_{app}=\zeta,\\
\nabla \cdot v_{app} =0,\\
v_{app|t=0} =\zeta_{ini},\\
v_{app,3|z=0}=0,\quad v_{app,h|z=1}= 0,\\
v_{app,3|z=1}=0,\quad \d_z v_{app,h|z=1}=\beta{\sigma}^\eps+\eta_1.
\end{aligned}
\end{equation}
with  $\zeta_{ini}\to 0$ in $L^2(\omega)$  and $\zeta \to 0$ in
$L^2([0,T]\times\omega_h)$.

In order to do so, we just apply Lemma \ref{stop} in the Appendix with
$$\delta^0_h=-\eps\eta_0,\quad \delta^0_3=0 \hbox{ and } \delta^1=0.$$
A simple computation allows then to establish all the properties
(\ref{forcing-app*}).

$\bullet$ The convergence is then obtained by a standard energy
estimate. Combining (\ref{forcing-app*}) and (\ref{forcing}), and
integrating by parts lead indeed to
$$\begin{aligned}
\frac12 \| (u_\eps-v_{app})(t)\|_{L^2}^2 +\int_0^t \|\nabla_h
(u_\eps-v_{app})\|_{L^2}^2 ds+\nu \int_0^t \|\nabla_h (u_\eps-v_{app})\|_{L^2}^2
ds\\
\leq \frac12 \|\zeta_{ini} \|_{L^2(\omega)}^2+ \int_0^t  \|
u_\eps-v_{app}\|_{L^2(\omega)} \| \zeta\|_{L^2(\omega)} ds \\+\nu \int_0^t
\| (u_\eps-v_{app})_{h|z=1} (t)\|_{L^2(\omega_h)}\|
\eta_1\|_{L^2(\omega_h)} ds
\end{aligned}
$$

$\bullet$ To conclude we therefore need to estimate the trace
$(u_\eps-v_{app})_{h|z=1}$ in $L^2(\omega_h)$ in terms of the $H^1$ norm of
$u_\eps-v_{app}$. By Sobolev embeddings and the Cauchy-Schwarz inequality,
we have
$$
\nu^{1/2} \| (u_\eps-v_{app})_{h|z=1} \|^2_{L^2(\omega_h)}\leq C \|
u_\eps-v_{app}\|^2_{L^2(\omega)} +\nu \| \d_z
(u_\eps-v_{app})\|^2_{L^2(\omega)}$$
Plugging that estimate in the previous energy inequality, we get
$$\begin{aligned}
\frac12 \| (u_\eps-v_{app})(t)\|_{L^2}^2 &+\int_0^t \|\nabla_h
(u_\eps-v_{app})\|_{L^2}^2 ds+\frac\nu2 \int_0^t \|\d_z
(u_\eps-v_{app})\|_{L^2}^2 ds\\
\leq &\frac12 \|\zeta_{ini} \|_{L^2(\omega)}^2+\frac12 \int_0^t
\|\zeta\|^2_{L^2(\omega)} ds +\nu^{3/2} \int_0^t \|
\eta_1\|^2_{L^2(\omega_h)} ds\\
&+ \frac12 \int_0^t \| (u_\eps-v_{app})(s)\|_{L^2(\omega)}^2ds
\end{aligned}
$$
using again the Cauchy-Schwarz inequality. We conclude by Gronwall's lemma
$$\begin{aligned}
\frac12 \| (u_\eps-v_{app})(t)\|_{L^2}^2 +\int_0^t \|\nabla_h
(u_\eps-v_{app})\|_{L^2}^2 ds+\nu \int_0^t \|\d_z (u_\eps-v_{app})\|_{L^2}^2
ds\\\leq \frac{e^{2Ct}}2 \|\zeta_{ini} \|_{L^2(\omega)}^2+\frac12
\int_0^t  \|\zeta\|^2_{L^2(\omega)} e^{2C(t-s)}ds +\nu^{3/2} \int_0^t
\|  \eta_1\|^2_{L^2(\omega_h)} e^{2C(t-s)}ds\\
\end{aligned}
$$
which proves that $u_\eps-v_{app}$ converges to 0 in
$L^\infty_{\text{loc}}(\bR_+, L^2(\om))$. Theorem \ref{thm} will be
proved in the case when $\gamma=0$ if we are able to build some
approximate solution $u_{app}$ that converges strongly to 0 as
$\eps,\nu\to 0$.
\end{proof}

\begin{Rmk}
The above proposition can be slightly modified if one wishes to work
with a source term $\eta$ belonging to $L^2([0,T], H^{-1}(\om))$, for
instance. In this case, following exactly the same argument as in the
proof above, the relevant assumption on $\eta$ is
\be
\label{hyp:etaH1}
\frac{1}{\sqrt{\nu}} \| \eta\|_{L^2([0,T], H^{-1}(\om))}=o(1)\quad
\text{as }\eps,\nu\to 0.
\ee

\label{rmk:energy1}
\end{Rmk}

\subsection{The first order terms of the approximate solution}

\label{par:approx1}

In order to obtain some approximate solution to (\ref{forcing}) (in
the sense  (\ref{forcing-app}) of the previous paragraph), we will
essentially need to construct the boundary layer term and some small
corrector to account for the vertical component of the boundary
condition.

\bigskip
\noindent
$\bullet$  We define with the notations of Proposition \ref{BL-thm}
$$u^{BL,1} =\cB (0,\beta\sigma)=\bar u^{BL,1}+\tilde
u^{BL,1}+u^{BL,1}_{res}\,.$$
Since we assume that $\sigma$ has a finite number of horizontal
Fourier modes $k_h$  and of oscillating modes $\mu$,
by Lemma \ref{BL-thm}, we have
$$
\begin{aligned}
\| \bar u^{BL,1}_h \|_{L^2(\omega)} \leq C\|
\sigma\|_{L^2(\omega_h)}\beta(\eps \nu)^{3/4},\\
\|\bar u^{BL,1}_3 \|_{L^2(\omega)}\leq C\| \sigma\|_{L^2(\omega_h)}
\beta(\eps \nu)^{5/4},\\
\end{aligned}
$$
and for the quasi-resonant modes with $|\mu|=1$, $k_h\neq 0$,
$$
\begin{aligned}
\| \tilde u^{BL,1}_h \|_{L^2(\omega)} \leq C\|
\sigma\|_{L^2(\omega_h)}\beta\left( \frac{\eps\nu}{\eps + \sqrt{\eps
\nu}} \right)^{3/4}\leq C \beta \nu^{3/4},\\
\|\tilde u^{BL,1}_3 \|_{L^2(\omega)}\leq C\|
\sigma\|_{L^2(\omega_h)}\beta\left( \frac{\eps\nu}{\eps + \sqrt{\eps
\nu}} \right)^{5/4}\leq C \beta \nu^{5/4}.\\
\end{aligned}
$$
As for the resonant modes $|\mu|=1$, $k_h=0$, we get
$$
\begin{aligned}
\| u^{BL,1}_{res, h} \|_{L^2([0,T]\times \omega)} \leq C \beta T^{5/4}
\| \sigma\|_{L^\infty([0,T], L^2(\omega_h))}\nu^{3/4},\\
u^{BL,1}_{res, 3}\equiv 0.
\end{aligned}
$$
Hence $u^{BL,1}$ vanishes provided
\be\label{cond_beta:1}
\beta \nu^{3/4}=o(1)\quad\text{as }\eps,\nu\to 0.
\ee

Furthermore, using the explicit formula for $\cB$, we get
$$
\begin{aligned}
\|\bar u^{BL,1}_{3|z=1}\|_{H^s(\omega_h)}=O(\beta(\eps \nu)) \hbox{
and } \|\d_t \bar u^{BL,1}_{3|z=1}\|_{H^s(\omega_h)}= O(\beta\nu)\\
\|\bar u^{BL,1}_{3|z=0}\|_{H^s(\omega_h) } =O(\beta(\eps \nu)^N)
\hbox{ and } \|\d_t \bar
u^{BL,1}_{3|z=0}\|_{H^s(\omega_h)}=O(\beta(\eps \nu)^N) ,\\
\|\tilde u^{BL,1}_{3|z=1}\|_{H^s(\omega_h)}=O(\beta(\eps \nu)^{1/2})
\hbox{ and } \|\d_t \bar u^{BL,1}_{3|z=1}\|_{H^s(\omega_h)}=
O\left(\beta\sqrt{\frac{\nu}{\eps}}\right)\\
\|\tilde u^{BL,1}_{3|z=0}\|_{H^s(\omega_h) } =O(\beta(\eps \nu)^N)
\hbox{ and } \|\d_t \bar
u^{BL,1}_{3|z=0}\|_{H^s(\omega_h)}=O(\beta(\eps \nu)^N)
\end{aligned}
$$
for any integer $N$, and uniformly in time. As a consequence, $\bar
u^{BL,1}_{3|z=1}$, $\bar u^{BL,1}_{3|z=0}$ and $\tilde
u^{BL,1}_{3|z=0}$ satisfy the conditions of the stopping Lemma
\ref{stop} in the Appendix as soon as $\beta \nu=o(1)$, which is
always ensured by hypothesis \eqref{cond_beta:1}. We  denote by $w$
the function defined in Lemma \ref{stop} with
$$
\begin{aligned}
\delta^0_h=0,\quad& \delta^1_h=0\\
\delta^0_3=-\bar u^{BL,1}_{3|z=0}-\tilde
u^{BL,1}_{3|z=0},\quad&\delta^1_3=-\bar u^{BL,1}_{3|z=1}.
\end{aligned}
$$

\noindent
$\bullet$  The term $\tilde u^{BL,1}_{3|z=1}$, on the other hand, does
not match the conditions of Lemma \ref{stop}. We therefore introduce
some  corrector $v^{int,1}$ to restore the zero-flux  condition. We
first define its vertical component
$$v^{int,1}_3=  -\tilde u^{BL,1}_{3|z=1}z\,,$$
then its horizontal component in order that the divergence-free
condition is satisfied
$$v^{int,1}_h =\nabla_h (\Delta_h)^{-1} \tilde u^{BL,1}_{3|z=1}.$$
Note that for $k_h=0$, $v^{int,1}$ is identically zero. In any case,
we get easily that
$$\begin{aligned}
   \|  v^{int,1}\|_{L^\infty([0,\infty),H^s(\omega))}
=O\left(\beta\frac{\eps\nu}{\eps + \sqrt{\eps\nu}}\right) = O(\beta
\nu),\\
\| \d_t  v^{int,1}\|_{L^\infty([0,\infty),L^2(\omega))}=O\left(\beta\sqrt{\frac{\nu}{\eps}}\right).
  \end{aligned}
$$
With the above notations, the first order of the approximate solution
is given by $$u_{app}^1=u^{BL,1} + w + v^{int,1}.$$

\subsection{Proof of Theorem \ref{thm} when there is no quasi-resonant mode}
\label{par:forcing-nonres}
If there is no quasi-resonant mode (see the precise definition in the
previous section), namely if $\tilde u^{BL,1}=0$, we then claim that
$u_{app}^1$ satisfies the required conditions.
We indeed have clearly
$$u^1_{app, 3|z=0} =u^1_{app,3|z=1} =0$$
by definition of $w$. Notice that in this case $v^{int,1}=0$. We further have
$$
\begin{aligned}
\d_z u^1_{app, h|z=1}- \beta\sigma^\eps=0,\\
\|u^1_{app,h|z=0}\|_{L^2(\omega_h)}=O(\beta(\eps \nu)^{N})\hbox{ and }
\|\d_t u^1_{app,h|z=0}\|_{L^2(\omega_h)}=O(\beta(\eps \nu)^{N})
\end{aligned}
$$
for all $N$.
We also have for  all $t\geq 0$
$$
\| u^1_{app} (t)\|_{L^2(\omega)} \leq \| u^{BL,1}(t)\|_{L^2(\omega)} +
 \| w(t)\|_{L^2(\omega)} =O(\beta(\eps \nu)^{3/4})=o(1).
$$

It remains then to check that the evolution equation is approximately satisfied.
We have
$$\d_t u_{app}^1 +\frac1\eps \bP(e_3\wedge u_{app}^1) -\Delta_h
u_{app}^1-\nu \d_{zz} u_{app}^1=O( \nu\beta)_{L^2([0,T]\times
\omega)}=o(1)$$
supplemented with some initial condition
$$u^1_{app|t=0} =O(\beta(\eps \nu)^{3/4})_{L^2(\omega)}.$$

We therefore apply Proposition \ref{energy1} and conclude that
$u^1_{app}$ has the same asymptotic behaviour as the solution of
\begin{equation}
\label{forcing*}
\begin{aligned}
\d_t u +\frac1\eps \bP(e_3\wedge u) -\Delta_h u-\nu \d_{zz} u=0,\\
\nabla \cdot u =0,\\
u_{|t=0} =0,\\
u_{3|z=0}=0,\quad u_{h|z=0}=0,\\
u_{3|z=1}=0,\quad \d_z u_{h|z=1}=\beta \sigma^\eps.
\end{aligned}
\end{equation}
Since $u^1_{app}$ vanishes in $L^\infty_{loc}(\bR_+, L^2(\om)),$
Theorem \ref{thm} is proved when $\gamma=0$ and when there is no
quasi-resonant mode in the forcing $\sigma.$

\subsection{Proof of the Theorem in the quasi-resonant case}

\label{par:forcing-res}

For the quasi-resonant modes $|\mu|=1$, the influence of the forcing
is much more extended inside the domain. In particular, the defect
$$
\begin{aligned}
\Sigma =&\d_t  v^{int,1} +\frac1\eps e_3\wedge  v^{int,1} -\Delta_h  v^{int,1}\\
=&\sum_{\mu=\pm 1} \sum _{k_h}\left( i\frac \mu\eps +|k_h|^2\right)
\hat v^{int,1} (\mu,k_h,z)e^{ik_h\cdot x_h}  e^{i\mu\frac{t}{\eps}}\\
&+\frac1\eps \sum_{\mu=\pm 1} \sum _{k_h}\begin{pmatrix}- \hat
v^{int,1}_2 (\mu,k_h,z)\\\hat v^{int,1}_1 (\mu,k_h,z)\\0 \end{pmatrix}
e^{ik_h\cdot x_h}  e^{i\mu\frac{t}{\eps}}
\end{aligned}
$$
does not converge strongly to 0 in $L^2$ norm. It is however expected
to have rapid oscillations, and thus to converge weakly to 0. The
standard method to deal with such a problem consists then in building
some corrector which will be small in $L^2$ norm in contrast with
its  time derivative which has to compensate the previous defect.

More precisely we will use the small divisor estimate stated in
Appendix~B. For $K>0$ arbitrary, denote by $\delta u^{int,1}_K=\sum_l
\hat w_l e^{-i\frac t\eps \lambda_l}N_l $ the solution to
$$ \d_t \delta u^{int,1}_K +\frac1\eps \bP(e_3\wedge \delta
u^{int,1}_K) -\Delta_h\delta u^{int,1}_K -\nu \d_{zz} \delta
u^{int,1}_K=-\bP_K (\Sigma),$$
supplemented with the initial condition
$$\delta u^{int,1}_{K|t=0}=0.$$
The notation $\bP_K$ stands for the projection onto the vector space
generated by $\{ N_l,|l|\leq K\}$. The idea is the to choose carefully
the truncation parameter $K$, depending on $\eps $ and $\nu$, so that
both $\delta u^{int,1}_K$ and the error term $\bP (\Sigma)-
\bP_K(\Sigma)$ are small in suitable Sobolev norms as $\eps$ and $\nu$
vanish.

\bigskip

\noindent $\bullet$ Let us first derive the equation on $\hat w_l$.
For $|l|\leq K$, $\hat w_l$ is the solution of
$$\d_t\hat  w_l +|l_h|^2\hat  w_l +\nu'|l_3|^2\hat w_l=-e^{i\lambda_l
\frac{t}{\eps}}\la N_l| \Sigma\ra $$ where $\nu'=\pi^2\nu$.
Direct computations give for $l_h \neq 0$, $\mu=\pm 1$,
$$
\begin{aligned}
 \hat v^{int,1}_h (\mu,l_h,z)=i\hat\delta_3(\mu,l_h)\frac{l_h}{|l_h|^2},\\
\hat v^{int,1}_3 (\mu,l_h,z)=\hat\delta_3(\mu,l_h)z,
\end{aligned}
$$
where
$$
\hat\delta_3(\mu,l_h)=i\beta(\eps\nu)\frac{\alpha^1_\mu(\mu,l_h)l_h\cdot
w_{\lambda^\mu}}{(\lambda^\mu)^2},
$$
where $\alpha^1_\mu$ and $w_\lambda$ were defined in the previous
section by \eqref{alpha-def} and \eqref{def:wlambda} respectively.
Notice moreover that $\lambda^\mu$ satisfies the estimates
\eqref{est:lambda+-1}-\eqref{est:lambda+-2}, so that in general,
$$
(\lambda^\mu)(\mu,k_h)^{-2}= O((\eps\nu)^{-1/2}).
$$

Moreover,
\begin{multline}
\label{scalarproduct}
\left\langle N_l\left| \begin{pmatrix} il_1\\il_2\\|l_h| ^2
z\end{pmatrix}e^{i l_h\cdot x_h}\right.\right\rangle= i{|l_h|^3\over
2\pi ^2|l|  l_3  } (-1)^{l_3}     \mathbf1_{l_3\neq 0}\\
\left\la N_l\left|\begin{pmatrix} - il_2\\il_1\\0\end{pmatrix}e^{i
l_h\cdot x_h}\right.\right\ra =\left\{\begin{array}{ll} 0&\text{ if
}l_3\neq 0,\\
\ds-\frac{|l_h|}{2\pi}&\text{ else.}\end{array}\right.
\end{multline}
We thus have
\begin{equation}\label{eq_wl}
\begin{aligned}
\d_t \hat w_l + (|l_h|^2+\nu' |l_3|^2)\hat w_l\\
= \frac{1}{\eps}\sum_{\mu=\pm 1}\frac{\hat
\delta_{3}(\mu,l_h)}{2\pi}\left( \mathbf1_{l_3\neq 0}\frac{(\mu-i\eps
|l_h|^2)|l_h|}{\pi |l| l_3} + \frac{\mathbf 1_{l_3=0}}{|l_h|}
\right)e^{i(\lambda_l+\mu)\frac{t}{\eps}}.
\end{aligned}
\end{equation}

\smallskip

\noindent $\bullet$ We now estimate the different terms and explain
how to choose the truncation parameter $K$. Notice first that by
truncating the large frequencies in $l$, we have introduced a source
term in the equation. Precisely, $\delta u^{int,1}_K + v^{int,1}$ is a
solution of equation \eqref{NS1} with a source term equal to
$$
(\Sigma-\bP\Sigma)+(\bP\Sigma -\bP_K\Sigma).
$$
The term $\Sigma-\bP\Sigma$ belongs to $V_0^{\bot}$ by definition of
$\bP$, and thus for all $u\in V_0$, we have
$$
\int_{\om}(\Sigma-\bP\Sigma)\cdot u=0.
$$
As for the remainder term $\bP\Sigma -\bP_K\Sigma$, we have
\be\label{est:resteK}\begin{aligned}
  \|\bP\Sigma -\bP_K\Sigma \|_{L^{\infty}((0,\infty),L^2(\om))}\leq C
\beta\sqrt{\frac{\nu}{\eps}}{K^{-3/2}},\\
  \|\bP\Sigma -\bP_K\Sigma \|_{L^{\infty}((0,\infty),H^{-1}(\om))}\leq
C  \beta\sqrt{\frac{\nu}{\eps}}{K^{-5/2}} .
                     \end{aligned}
\ee
With a view to apply Proposition \ref{energy1}, or its variant
sketched in Remark \ref{rmk:energy1}, we need the source term $\bP
\Sigma- \bP_K\sigma$ to be either $o(1)$ in $L^2$ norm or
$o(\sqrt{\nu})$ in $H^{-1}$ norm as $\eps,\nu\to 0$ (see condition
\eqref{hyp:etaH1}). Precisely, according to Proposition \ref{energy1}
and Remark \ref{rmk:energy1}, the parameter $K$ should satisfy either
\be
\beta \sqrt{\frac{\nu}{\eps}} K^{-3/2}=o(1)\quad\text{as }\eps,\nu\to
0\, ,\label{condK1}
\ee
or
\be
\frac{1}{\sqrt{\nu}}\beta \sqrt{\frac{\nu}{\eps}}
K^{-5/2}=\frac{\beta}{\sqrt{\eps}} K^{-5/2}=o(1)\quad\text{as
}\eps,\nu\to 0\, .\label{condK2}
\ee

On the other hand, we  apply Lemma \ref{smalldivisor} to get
$$
\|\delta u^{int,1}_K\|_{H^s(\omega)}\leq C
\beta(\eps\nu)^{\frac{1}{2}}K^{s+\frac{1}{2}}.$$
For further purposes, we have to choose $K$ such that the $H^s$ norm
of $\delta u^{int,1}_K$ satisfies
$$
\sqrt{\frac{\nu}{\eps}}\|\delta
u^{int,1}_K\|_{H^s(\omega)}=o(1)\quad\text{as }\eps,\nu\to 0\, ,
$$
and such that at least either \eqref{condK1} or \eqref{condK2} is
satisfied. We distinguish between the cases when $\nu$ is large (say
$\nu\geq \eps$) and $\nu$ is small (say $\nu\leq \eps $), which yield
different values for $K$.

\noindent- \label{page:hypbeta} If $\nu\leq \eps$, we choose $K$ so that $$
\beta\sqrt{\frac{\nu}{\eps}}{K^{-3/2}}=\beta\sqrt{\frac{\nu}{\eps}}(\eps\nu)^{1/2}K^{s+\frac{1}{2}}
$$
for some $s>3/2$, which yields
$$
K=(\eps\nu)^{-\frac{1}{2(s+2)}}.
$$
With this choice, we have
$$
\|\bP\Sigma -\bP_K\Sigma \|_{L^2},\  \sqrt{\frac{\nu}{\eps}}\|\delta
u^{int,1}_K\|_{H^s(\omega)}\leq C \beta \nu^{1-
\frac{s+\frac{1}{2}}{2(s+2)}}\eps^{-\frac{s+\frac{1}{2}}{2(s+2)}};
$$
Now, assume that $\beta$ satisfies the following assumption
\be\label{cond_beta:2}
\begin{aligned}
 \exists (\alpha_0,\alpha_1)\in(0,\infty)^2,\ \alpha_0 <5/7\text{ and
}\alpha_1>2/7, \ \exists C>0,\\
\nu\leq \eps \Rightarrow \beta \leq C \nu^{-\alpha_0} \eps^{\alpha_1}.
\end{aligned}
\ee
We  choose $s_0>3/2$ such that
$$\begin{aligned}
   1- \frac{s_0+\frac{1}{2}}{2(s_0+2)}-\alpha_0>0,\\
\alpha_1-\frac{s_0+\frac{1}{2}}{2(s_0+2)}>0,
  \end{aligned}
$$
and we have, as $\eps,\nu\to 0$,\be\label{est:wK1}
\|\bP\Sigma -\bP_K\Sigma \|_{L^2}+\ \sqrt{\frac{\nu}{\eps}} \|\delta
u^{int,1}_K\|_{H^{s_0}(\omega)}=o(1).
\ee

\noindent- Else, we choose $K$ so that
$$
\beta\frac{1}{\sqrt{\eps}}{K^{-5/2}}=\beta\nu K^{s+\frac{1}{2}}
$$
for some $s>3/2$, which yields
$$
K=(\nu\sqrt{\eps})^{-\frac{1}{s+3}}.
$$
Assume now that $\beta$ satisfies the following assumption
\be\label{cond_beta:3}
\begin{aligned}
 \exists (\alpha_0,\alpha_1)\in(0,\infty)^2,\ \alpha_0 <5/9\text{ and
}\alpha_1>2/9, \ \exists C>0,\\
\nu\geq \eps \Rightarrow \beta \leq C \nu^{-\alpha_0} \eps^{\alpha_1}.
\end{aligned}
\ee
We then choose $s_0>3/2$ such that
$$
\begin{aligned}
\alpha_1 - \frac{s_0+ \frac{1}{2}}{2(s_0+3)}>0,\\
1-\alpha_0 - \frac{s_0+ \frac{1}{2}}{s_0+3}>0,
\end{aligned}
$$
and we have, as $\eps,\nu\to 0$,\be\label{est:wK2}
\frac{1}{\sqrt{\nu}}\|\bP\Sigma -\bP_K\Sigma
\|_{L^{\infty}((0,\infty),H^{-1}(\om))}+
\sqrt{\frac{\nu}{\eps}}\|\delta
u^{int,1}_K\|_{L^{\infty}((0,\infty),H^{s_0}(\omega))}=o(1).
\ee

We emphasize that this method remains valid when $\nu=O(\eps)$;
however, if $\nu=\eps$,  condition \eqref{cond_beta:3} is more
restrictive than \eqref{cond_beta:2}.

\medskip
$\bullet$ Because of the terms $v^{int,1}$ and $\delta u^{int,1}_K$,
the horizontal boundary conditions are no longer satisfied at $z=0$
(notice however that they are satisfied at $z=1$). Thus, we construct
another boundary layer term, which we denote by $\delta u^{BL,1}$,
such that
$$
\delta u^{BL,1}=\mathcal B(-v^{int,1}_{h|z=0} -\delta u^{int,1}_{K,h|z=0},0 ).
$$
The above definition is not entirely licit, since $\delta
u^{int,1}_{K,h|z=0}$ takes the form
$$
\delta u^{int,1}_{K,h|z=0}(t,x_h)=\sum_{|l|\leq K}\hat w_l(t)
e^{-i\lambda_l\frac{t}{\eps}}e^{ik_h\cdot x_h}\begin{pmatrix}

                                      n_1(k)\\n_2(k)

                                     \end{pmatrix},
$$
where the vector $n(k)$ is defined in Appendix A (see
\eqref{base1}-\eqref{base2}. Hence $\delta u^{int,1}_{K,h|z=0}$
depends on the fast time variable $t/\eps$, but also on the slow time
variable $t $ through the coefficient $\hat w_l$. In the definition of
$\delta u^{BL,1}$, we forget the time dependance of $\hat w_l$, and
consider the coefficients $\hat w_l$ as constants. Consequently, the
boundary layer term $\delta u^{BL,1} $ is not an exact solution of
equation \eqref{NS1}, but there is an error term depending on $\d_t
\hat w_l$. This error term will be estimated later on.

We now decompose $\delta u^{BL,1}$ into $\delta u^{BL,1}=\delta \tilde
u^{BL,1} + \delta \bar u^{BL,1}$ as in Lemma \ref{BL-thm}; the term
$\delta \tilde u^{BL,1}$ is due to the modes $k_h\neq 0$, $|\mu|=1$,
and thus depends only on $v^{int,1}$, since $|\lambda_k|< 1$ if
$k_h\neq 0$. Notice that there is no term $\delta u^{BL,1}_{res}$
because $\hat v^{int,1}(\mu,l_h,z)=0$ for $k_h=0$, $\hat w_l=0$ for
$k_h=0$.

According to the estimates \eqref{B-cont}, and provided
\eqref{cond_beta:3} holds, we have, for all $t>0$,
\begin{eqnarray*}
\| \delta u^{BL,1}_h(t)\|_{L^2(\om)}&\leq & C \|
v^{int,1}(t)\|_{L^2(\om)}\left(\eps\nu\right)^{\frac{1}{8}}+C \|
\delta u^{int,1}_K(t)\|_{H^{s_0}} (\eps\nu)^{1/4} \\
&\leq &C (\eps\nu)^{1/8},\\
\| \delta u^{BL,1}_3(t)\|_{L^2(\om)}&\leq &\|
v^{int,1}(t)\|_{L^2(\om)}\left(\eps\nu\right)^{\frac{3}{8}}+ C \|
\delta u^{int,1}_K(t)\|_{H^{s_0+1}} (\eps\nu)^{3/4}\\
&\leq & C (\eps\nu)^{3/8} + C (\eps\nu)^{3/4}
({\nu}\sqrt\eps)^{-\frac{1}{s_0+3}}.
\end{eqnarray*}
Thus $\delta u^{BL,1}$ vanishes in $L^\infty([0,T],L^2(\om) )$.

Let us now estimate the error term in equation \eqref{NS1} due to the
time dependance of $\hat w_l$.
According to \eqref{eq_wl}, there exists a constant $C$ such that
$$
|\d_t w_l|\leq \frac{C}{|l_3|^2}\sqrt{\frac{\nu}{\eps}}\beta,
$$
so that $\delta u^{BL,1} $ is an approximate solution of equation
\eqref{NS1}, with an error term which is bounded from above in
$L^2([0,T]\times \omega)$ by
$$
\sqrt{\frac{\nu}{\eps}}\beta(\eps\nu)^{1/4}= \nu^{3/4}\eps^{-1/4}\beta.
$$
Hence, the new condition on $\beta$ is
\be\label{cond_beta:4}
\beta=o\left( \nu^{-3/4}\eps^{1/4} \right)\quad\text{as }\eps,\nu\to 0.
\ee
Notice that \eqref{cond_beta:4} immediately entails \eqref{cond_beta:1}.

\noindent $\bullet$ Let us now check that the remaining boundary terms
are all sufficiently small to conclude. To begin with,
the terms $\delta u^{BL,1}_{h|z=1}$,  $\delta u^{BL,1}_{3|z=1}$ are
exponentially small, and thus satisfy the hypotheses of Proposition
\ref{energy1} and Lemma \ref{stop} respectively. We now prove that
under conditions \eqref{cond_beta:2}-\eqref{cond_beta:3}, $\delta \bar
u^{BL,1}_{3|z=0}$ also satisfies the assumptions of Lemma \ref{stop}.
Using the construction of the previous section, it can be checked that
$\delta \bar u^{BL,1}$ is given by
\begin{eqnarray*}
\delta \bar u^{BL,1}(t)&=&\sum_{|k_h|\leq N}\sum_{|k_3|\leq
K}\sum_{\mu\in\{-1,1\}} e^{-i\lambda_k \frac{t}{\eps}}
\alpha^0_\mu(-\lambda_k,k_h) W^0_{\lambda^\mu}\\
&+&\sum_{|k_h|\leq
N}\sum_{\mu\in\{-1,1\}}e^{i\mu\frac{t}{\eps}}\alpha^0_{-\mu}(\mu,k_h)W^0_{\lambda^{-\mu}}
\end{eqnarray*}
where  the coefficients $\alpha^0_\mu$ satisfy
$$\begin{aligned}
   \forall k\in\bZ^3,\quad \left| \alpha_\mu^0(-\lambda_k,k_h)
\right|\leq C |\hat w_k|,\\
\text{and }\left| \alpha_{-\mu}^0(\mu,k_h) \right|\leq C \left| \hat
\delta_3(\mu,k_h) \right|.
  \end{aligned}
$$
Recalling the expression of $W^0_\lambda$ (see \eqref{vlambda-def}),
we infer that for all $t,x_h$
\begin{eqnarray*}
 \left| \delta \bar u_{3|z=0}^{BL,1} (t,x_h)\right|&\leq& C
\sum_{|k_h|\leq N}\sum_{|k_3|\leq K}\sum_{\mu\in\{-1,1\}} |\hat
w_k|\frac{\sqrt{\eps\nu}}{|\lambda^\mu(-\lambda_k,k_h)|}\\
&+&C \sqrt{\eps\nu }  \sum_{|k_h|\leq N}\sum_{\mu\in\{-1,1\}}\left|
\hat \delta_3(\mu,k_h) \right| \\
&\leq & C\sqrt{\eps\nu} \sum_{|k_h|\leq N}\sum_{|k_3|\leq K} |k| |\hat
w_k| + C\beta \eps \nu,
\end{eqnarray*}
and thus, using the Cauchy-Schwarz inequality (recall that $s_0>3/2$
and that $N$ is bounded)
\begin{eqnarray*}
\left\|  \delta \bar u_{3|z=0}^{BL,1}\right\|_{L^\infty([0,T],
L^2(\omega_h))}&\leq & C
\sqrt{\eps\nu}\sup_{t\in[0,T]}\left[\sum_{\substack{|k_h|\leq
N,\\|k_3|\leq K}}|k|^{2s_0} |\hat w_k(t)|^2\right]^{1/2} + C\beta \eps
\nu\\
&\leq & C \sqrt{\eps\nu}\left\| \delta u^{int,1}_K
\right\|_{L^\infty([0,T], H^{s_0})} + C\beta \eps\nu.
\end{eqnarray*}
Hence, under conditions \eqref{cond_beta:2}-\eqref{cond_beta:3} and by
definition of $K$, the remaining boundary term $\delta \bar
u_{3|z=0}^{BL,1}$ satisfies the conditions of Lemma \ref{stop}.

We now consider $\delta \tilde u^{BL,1}_{3|z=0}$, which is due to the
modes $\mu=\pm 1$, $k_h\neq 0$ in $v^{int,1}$; we have
\begin{eqnarray*}
 \|\delta \tilde u^{BL,1}_{3|z=0} \|_{L^\infty([0,T],
L^2(\om_h))}&\leq& C \frac{\sqrt{\eps\nu}}{(\eps\nu)^{1/4}+
\sqrt{\eps}} \| v^{int,1}\|_{L^\infty([0,T],L^2(\om))}\\
&\leq & C \beta\left( \frac{{\eps\nu}}{(\eps\nu)^{1/2}+ {\eps}}
\right)^{3/2}\|\sigma \|_{L^\infty([0,T], L^2(\om_h))}\\&\leq &C \beta
(\eps\nu)^{3/4}.
\end{eqnarray*}
Hence $\delta \tilde u^{BL,1}_{3|z=0}$ satisfies the assumptions of
Lemma \ref{stop}, provided \eqref{cond_beta:4} is satisfied.

Thus, we slightly modify the definition of the function $w$ given by
Lemma \ref{stop}, so that the boundary conditions are now
$$
\begin{aligned}
\delta^0_h=0,\quad& \delta^1_h=0\\
\delta^0_3=- u^{BL,1}_{3|z=0}-\delta
u^{BL,1}_{3|z=0},\quad&\delta^1_3=-\bar u^{BL,1}_{3|z=1}-\delta
u^{BL,1}_{3|z=1}.\end{aligned}
$$

\medskip
$\bullet$
We then claim that under hypotheses  \eqref{cond_beta:2},
\eqref{cond_beta:3} and \eqref{cond_beta:4}, $$u_{app} = u^{BL,1}+ w+
v^{int,1}+\delta u^{int,1}_K + \delta u^{BL,1}$$ satisfies the
assumptions of Proposition \ref{energy1}.
We indeed have clearly
$$u_{app, 3|z=0} =u_{app,3|z=1} =0$$
by definition of $v^{int,1}$ and $w$. We further have, for all $N>0$
$$
\begin{aligned}
\|\d_z u_{app, h|z=1}- \beta\sigma^\eps\|_{L^2(\omega_h)}=O((\eps \nu)^N),\\
\|u_{app,h|z=0}\|_{L^2(\omega_h)}=O((\eps \nu)^N)\hbox{ and } \|\d_t
u_{app,h|z=0}\|_{L^2(\omega_h)}=O((\eps \nu)^N).
\end{aligned}
$$
We also have for  all $t\geq 0$
\begin{eqnarray*}
 \| u_{app} (t)\|_{L^2(\omega)} &\leq & C\left( \beta \nu^{3/4}+
\beta(\eps\nu)^{1/2} (\sqrt{\nu}\eps)^{-\frac{1}{2s_0+6}}\right)=o(1).
\end{eqnarray*}

By definition of the different terms, the evolution equation is
approximately satisfied, up to an error term of order $o(\sqrt{\nu})$
in $L^\infty((0,\infty), H^1(\om))$, and another one of order $o(1)$
in $L^2((0,T)\times\omega).$

We therefore apply the variant of Proposition \ref{energy1} sketched
in Remark \ref{rmk:energy1} and conclude that $u_{app}$ has the same
asymptotic behaviour as the solution of
\begin{equation}
\label{forcing**}
\begin{aligned}
\d_t u +\frac1\eps \bP(e_3\wedge u) -\Delta_h u-\nu \d_{zz} u=0,\\
\nabla \cdot u =0,\\
u_{|t=0} =0,\\
u_{3|z=0}=0,\quad u_{h|z=0}=  0,\\
u_{3|z=1}=0,\quad \d_z u_{h|z=1}=\beta \sigma^\eps.
\end{aligned}
\end{equation}
Thus the solution of \eqref{forcing**} vanishes in
$L^\infty([0,T],L^2(\omega))$ norm as $\eps,\nu\to 0$ with
$(\eps,\nu,\beta)$ satisfying \eqref{cond_beta:2}, \eqref{cond_beta:3}
and \eqref{cond_beta:4}.

\medskip

\noindent $\bullet$ We conclude this paragraph by giving a  scaling
assumption on $\beta$ which entails all three conditions
\eqref{cond_beta:2}, \eqref{cond_beta:3} and \eqref{cond_beta:4}.
Assume that the parameter $\beta$ is such that
\be
\label{hyp:scaling}\exists \alpha_0\in \left( 0,\frac{7}{12} \right),\
\beta= O(\nu^{-\alpha_0}\eps^{1/4})\quad \text{as }\eps,\nu\to 0;
\ee
we now check that each of the assumptions \eqref{cond_beta:2},
\eqref{cond_beta:3} and \eqref{cond_beta:4} are satisfied.

First, it is obvious that
$$
\nu^{3/4}\eps^{-1/4}\beta= O(\nu^{3/4-\alpha_0})=o(1)
$$
since $3/4 - \alpha_0>1/6>0$. Hence \eqref{cond_beta:4} is satisfied.

We now tackle condition \eqref{cond_beta:2}; since $\alpha_0<7/12$,
there exists positive numbers $(\alpha_0',\alpha'_1)$ such that
$$
\alpha_0'<5/7,\ \alpha_1'>2/7,\ \text{and }
\alpha_0'-\alpha_1'=\alpha_0 - \frac{1}{4}.
$$
In view of \eqref{hyp:scaling}, there exists a constant $C$ such that
\begin{eqnarray*}
\beta&\leq & C \nu^{-\alpha_0' + \alpha_1' - \frac{1}{4}}\eps^{\frac{1}{4}} \\
&\leq & C \nu^{-\alpha_0'}\left( \frac{\eps}{\nu}
\right)^{\frac{1}{4}-\alpha_1'}\eps^{\alpha_1'}.
\end{eqnarray*}
Notice that $\alpha_1'>1/4$, and thus if $\nu\leq \eps$, we deduce that
$$
\beta \leq C \nu^{-\alpha_0'}\eps^{\alpha_1'}.
$$
Hence we have proved that \eqref{hyp:scaling}$\Rightarrow$ \eqref{cond_beta:2}.

The treatment of \eqref{cond_beta:3} is similar. We first choose
positive numbers $\alpha_0'',\alpha_1''$ such that
$$
\alpha_0''<5/9,\ 2/9<\alpha_1''<1/4,\ \text{and }
\alpha_0''-\alpha_1''=\alpha_0 - \frac{1}{4}.
$$
Then if $\nu\geq \eps$, we have
\begin{eqnarray*}
\beta &\leq & C  \nu^{-\alpha_0'' + \alpha_1'' -
\frac{1}{4}}\eps^{\frac{1}{4}} \\
&\leq & C\nu^{-\alpha_0''}\left( \frac{\eps}{\nu}
\right)^{\frac{1}{4}-\alpha_1''}\eps^{\alpha_1''}\\
&\leq & C \nu^{-\alpha_0''}\eps^{\alpha_1''}.
\end{eqnarray*}
Hence we also have \eqref{hyp:scaling}$\Rightarrow$
\eqref{cond_beta:3}, and eventually, we deduce that under hypothesis
\eqref{hyp:scaling}, the solution of \eqref{forcing} converges towards
zero in $L^\infty([0,T], L^2)$ for all $T>0$.

\section{Study of the dissipating part of the motion}\label{int+BL-par}

This section is dedicated to the rest of the proof of Theorem
\ref{thm}. According to the preceding section, there remains to define
the term $u^{Dirichlet}$, which is an approximate solution of
\eqref{NS1}, supplemented with the following boundary conditions
$$
\begin{aligned}
u^{Dirichlet}_{h|z=0}=0,\quad&u^{Dirichlet}_{3|z=0}=0\\
\d_zu^{Dirichlet}_{h|z=1}=0,&u^{Dirichlet}_{3|z=1}=0,\\
u^{Dirichlet}_{|t=0}=\gamma.&
\end{aligned}
$$

This point has already been investigated by several authors, see for
instance \cite{CDGG}: the idea is to construct an interior term,
denoted by $u^{int}$, which satisfies the evolution equation up to
error terms which are $o(1)$, and a boundary layer term, denoted by
$u^{BL}$, which restores the horizontal boundary conditions violated
by the interior term. We emphasize that in order that the equation and
the boundary conditions are satisfied up to sufficiently small error
terms, we need to build some second order terms in both $u^{int}$ and
$u^{BL}$.

The organization of the section is as follows: in the spirit of
Theorem \ref{BL-thm} and Definition \ref{BL-def}, we first define an
operator $\mathcal U$, which allows us to construct an interior term,
given arbitrary vertical boundary conditions. Then we explain how to
choose the boundary conditions for the boundary layer term and the
interior term in order to retrieve \eqref{top} and \eqref{bottom} with
$\sigma\equiv 0$. In the last paragraph, we build one additional
boundary layer term, and we prove Theorem \ref{thm} thanks to an
energy estimate.

Throughout this section, we use repeatedly the following norm: if
$\delta\in L^\infty([0,\infty)\times[0,\infty), L^2(\omega_h))$ is
such that
$$
\delta(t,\tau,x_h)=\sum_{|k_h|\leq N}\sum_{k_3\in\bZ}\hat
\delta(-\lambda_k,k_h;t)e^{ik_h\cdot x_h}e^{-i\lambda_k\tau},
$$
where $\tau$ stands for the fast time variable $t/\eps$, then
$$
\|\delta(t,\cdot)\|_s:= \left( \sum_{|k_h|\leq
N}\sum_{k_3\in\bZ}|k_3|^{2s}\left|\hat \delta(-\lambda_k,k_h;t)
\right|^2\right)^{1/2}.
$$

\subsection{ Construction of the operator $\mathcal U$}

Let $\delta^1_3$ and $\delta^0_3$ in
$L^\infty([0,\infty)\times[0,\infty), L^2(\omega_h))$ be such that
\begin{eqnarray}
 \label{def_delta3}
 \delta^j_3(t,\tau,x_h)&=&\sum_{|k_h|\leq N}\sum_{k_3\in\bZ}\hat
\delta^j_3(-\lambda_k,k_h;t)e^{ik_h\cdot x_h}e^{-i\lambda_k\tau},
\end{eqnarray}
and let $\gamma\in V_0$. In practice, the functions $\delta^1_3$ and
$\delta^0_3$ will not be arbitrary, and will be dictated by the
expression of the boundary layer operator constructed in the third
section. In fact, we will see that $\delta^1_3=0$, so that the
expression of $u^{int}$ below is simpler, but we have preferred to
keep  an arbitrary value for $\delta^1_3$ in order not to anticipate
on this result.

We define the operator $\mathcal U$ by
$$
\mathcal U(\gamma;\delta^0_3,\delta^1_3)=u^{int},
$$
where $u^{int}$ is an approximate solution of equation \eqref{NS1} and
satisfies the following boundary conditions
\begin{eqnarray}
u^{int}_{3|z=1}&=&\sqrt{\eps\nu}\delta^1_3,\label{CL:uint2}\\
u^{int}_{3|z=0}&=&\sqrt{\eps\nu}\delta^0_3,\label{CL:uint3}\\
u^{int}_{|t=0}&=&\gamma+ o(1).\label{CI:uint}
\end{eqnarray}
We emphasize that conditions \eqref{CL:uint2}-\eqref{CL:uint3} will be
satisfied exactly (without any error term). Of course the above
conditions are not sufficient to define the term $u^{int}$
unequivocally. We merely define here a particular solution of this
system, which is sufficient for our purposes.

The explicit construction of $u^{int}$ requires three steps: first, we
exhibit a divergence-free vector field $v^{int,0}$ which satisfies the
vertical boundary conditions \eqref{CL:uint2}-\eqref{CL:uint3}, but
not equation \eqref{NS1}, and then we define a function $ \delta
u^{int,0}$, which satisfies homogeneous boundary conditions, and such
that \be  u^{int}:=\exp\left( -\frac{t}{\eps} \right)u^{int}_L+\delta
u^{int,0}+ v^{int,0} \label{def:uint} \ee is an approximate solution
of \eqref{NS1}, supplemented with the initial condition
\eqref{CI:uint}. As usual in this type of problem, we first assume
that $\exp\left( -t/\eps \right)u^{int}_L$ is the preponderant term in
$u^{int}$, and thus we begin by deriving an equation for the corrector
 term $\delta u^{int,0}$ involving $ u^{int}_L$. Ultimately, this will
allow us to write an equation for $ u^{int}_L$.
In the third step, we prove that the function $\delta u^{int,0}$ thus
defined is of order $O(\sqrt{\nu\eps})$ in $L^2$.

$\bullet$ A natural choice for $v^{int,0}$ is
\be\left\{\begin{array}{l}
v^{int,0}_3 = \sqrt{\eps\nu}\left[\delta^1_3z + \delta^0_3(1-z)\right],\\
 v^{int,0}_h =\sqrt{\eps\nu}\ds \nabla_h \Delta_h^{-1}\left[\delta^0_3
- \delta^1_3\right].
 \end{array}\right.
 \label{def:vint}\ee
(Note that $v^{int,0}$ is not uniquely determined by
\eqref{CL:uint2}-\eqref{CL:uint3}). We denote by $\hat
v^{int,0}(\mu,k_h,t,z) $ the Fourier coefficient of $ v^{int,0}$, that
is
 $$v^{int,0}(t,x)=\sum_{\mu,k_h } \hat v^{int,0}(\mu,k_h,t,z) \exp
(ik_h\cdot x_h)\exp  \left(i\frac t\eps \mu\right) .$$
  The fact that $v^{int,0}_{3}\neq 0$  means  that a small amount of
fluid, of order $\sqrt{\eps\nu}\delta^j_3$, enters the domain (or the
boundary layer, depending on the sign of the coefficient). This
phenomenon is called Ekman suction and $v^{int,0}_3$ is called Ekman
transpiration velocity. This velocity will be responsible for global
circulation in the whole domain, of order $(\eps \nu)^\frac{1}{2}$,
but not limited to the boundary layer.

 Furthermore the Ekman suction at the bottom has a very important
effect in the energy balance. The order of magnitude of $\nu \int
|\nabla u^{BL}|^2$ in the Ekman layer is indeed $O(\sqrt{\nu\over
\eps})$, so that  the Ekman layer damps the interior motion, like a
friction term. This phenomenon is called {\bf Ekman pumping}. We
therefore expect that the weak limit flow of \eqref{NS1} in the high
rotation limit is not determined by the formal equations
(\ref{formal-mean}) but by a dissipative versions of this equation.

$\bullet$ As in the previous section, we seek
\begin{eqnarray}
\exp\left( -\frac{t}{\eps} L
\right)u^{int}_L&=&\sum_{l\in\bZ^3}c_l(t)e^{-i\lambda_l
\frac{t}{\eps}}N_l,\label{def:buint}\\
\delta u^{int,0}&=&\sum_{l\in\bZ^3}\delta c_l(t) e^{-i\lambda_l
\frac{t}{\eps}}N_l,\label{def:delta_uint}
\end{eqnarray}
so that
\begin{eqnarray*}
&&\left[ \d_t + \frac{1}{\eps} L -\nu \d_{zz} - \Delta_h \right]
\left( \exp\left( -\frac{t}{\eps} L \right)u^{int}_L+\delta u^{int,0}
\right) \\
&= &\sum_{l\in\bZ^3}\d_t (c_l(t)+\delta c_l(t))
e^{-i\lambda_l\frac{t}{\eps}}N_l\\&&+\sum_{l\in\bZ^3} (|l_h|^2 + \nu'
|l_3|^2)(c_l(t)+\delta c_l(t))e^{-i\lambda_l\frac{t}{\eps}}N_l\,,
\end{eqnarray*}
where $\nu'=\pi^2 \nu$.

On the other hand,
\begin{eqnarray*}
&& \left[ \d_t + \frac{1}{\eps} e_3\wedge -\nu \d_{zz} - \Delta_h
\right]v^{int,0}\\
&=&\sum_{\mu,k_h}\left[ \d_t\hat v^{int,0}(\mu,k_h,t,z) + |k_h|^2 \hat
v^{int,0}(\mu,k_h,t,z)  \right] e^{i k_h\cdot x_h}
e^{i\mu\frac{t}{\eps}}\\
&+& \frac{1}{\eps}\sum_{\mu,k_h} i\mu\hat v^{int,0}(\mu,k_h,t,z)  e^{i
k_h\cdot x_h} e^{i\mu\frac{t}{\eps}} \\
&
+&\sqrt{\frac{\nu}{\eps}}
\sum_{\mu,k_h}\frac{(\hat\delta^1_3-\hat\delta^0_3)(\mu,k_h,t)
}{|k_h|^2}\begin{pmatrix}

                         - ik_2\\ik_1\\0

                         \end{pmatrix}
e^{i k_h\cdot x_h} e^{i\mu\frac{t}{\eps}}.
\end{eqnarray*}
In order that $\exp(-t/\eps) L)u^{int}_L +\delta u^{int,0} +
v^{int,0}$ is an approximate solution of \eqref{NS1}, we project both
equations on $N_l$ for $l\in\bZ^3$, multiply by
$\exp(i\lambda_l\frac{t}{\eps})$, and identify each term. We further
apply the following rules in order to determine the equations for
$\delta u^{int,0}$ and $u^{int}_L$:
\begin{itemize}
\item all the terms which do not have fast oscillations and are of
order $ O(\delta^j_3\sqrt{\frac{\nu}{\eps}})$ become source terms in
the equation on $c_l$,
 \item all the terms which are either
$o(\delta^j_3\sqrt{\frac{\nu}{\eps}})$ or oscillating at a frequency
$1/\eps$ become source terms in the equation on $\delta c_l$.
\end{itemize}
We work with a fixed $l\in\bZ^3$. Recall that $v^{int,0}$ has no
purely vertical component, i.e. $\hat v^{int,0}(\mu,l_h,t,z)=0$ if
$l_h=0$.
Thanks to  formulas (\ref{scalarproduct}), the equation on $c_l$ reads
\be
\d_t c_l + |l_h|^2 c_l + \nu' |l_3|^2 c_l=
-\sqrt{\frac{\nu}{\eps}}\frac{|l_h|}{2\pi|l|^2}\left[
\hat\delta^0_3(-\lambda_l,l_h,t)-(-1)^{l_3}\hat\delta^1_3(-\lambda_l,l_h,t)\right]\label{eq:u_0},
\ee
supplemented with the initial condition
\be
c_l(0)=\la N_l|\gamma\ra\label{CIc_l},
\ee
and the equation on $\delta c_l$ is
\begin{eqnarray}
&&\d_t \delta c_l + (|l_h|^2+\nu' |l_3|^2)\delta c_l \label{eq:delta_c_l}\\
&=&- \sum_{\mu\neq -\lambda_l}\left\la N_l|(\d_t\hat v^{int,0}
(\mu,l_h,t,z)+ |l_h|^2 \hat v^{int,0} (\mu,l_h,t,z) )e^{il_h\cdot
x_h}\right\ra  e^{i(\lambda_l+\mu)\frac{t}{\eps}}\nonumber\\
&& \nonumber- \sqrt{\frac{\nu}{\eps}}\sum_{\mu\neq-
\lambda_l}\frac{\hat\delta^0_3(\mu,l_h,t)-(-1)^{l_3}\hat\delta^1_3(\mu,l_h,t)}{2\pi}
\times\\&&\qquad\qquad\qquad\qquad\qquad\qquad\qquad\times\left(
\mathbf1_{l_3\neq 0}\frac{\mu|l_h|}{\pi |l| l_3} + \frac{\mathbf
1_{l_3=0}}{|l_h|} \right)e^{i(\lambda_l+\mu)\frac{t}{\eps}}.\nonumber
\end{eqnarray}
For the time being, we do not specify an initial condition for $\delta
c_l.$ Indeed, we shall see that it is convenient to choose another
condition than $-\la N_l, v^{int,0}\ra$, in order to use the possible
decay  of $\hat \delta^j_3(\mu,l_h,t)$ with respect to $t$.
This choice will be made clear in paragraph \ref{par:corrector-D}.

As in the previous section, we  truncate the large frequencies in
$\delta c_l$. This creates an error term in the evolution equation,
which is of order
$$
O\left(\sqrt{\frac{\nu}{\eps}}\frac{1}{K^{3/2}}\right)_{L^2},
$$
where $K$ is the truncation parameter, to be chosen later on. We set
$$
 \delta u^{int,0}_K =\sum_{l_h}\sum_{|l_3|\leq K} \delta c_l N_l.
$$

$\bullet$ We now  apply to $\delta u^{int,0}_K$ the small divisor
estimate stated in Lemma \ref{smalldivisor} in the Appendix with
\begin{eqnarray*}
s(\mu,l,t) &=&-
\sqrt{\eps\nu}\left[\hat\delta^0_3(\mu,l_h,t)-(-1)^{l_3}\hat\delta^1_3(\mu,l_h,t)\right]\mathbf1_{l_3\neq
0}\frac{|l_h|^3}{\pi |l| l_3}\\ &-& \sqrt{\eps\nu}
\left[\d_t\hat\delta^0_3(\mu,l_h,t)-(-1)^{l_3}\d_t\hat\delta^1_3(\mu,l_h,t)\right]\mathbf1_{l_3\neq
0}\frac{ |l_h|}{\pi |l| l_3} \\
& - &\sqrt{\frac{\nu}{\eps}}\frac{\hat\delta^0_3(\mu,l_h,t)-(-1)^{l_3}\hat\delta^1_3(\mu,l_h,t)}{2\pi}\left(
\mathbf1_{l_3\neq 0}\frac{\mu |l_h|}{\pi |l| l_3} + \frac{\mathbf
1_{l_3=0}}{|l_h|} \right) ,
\end{eqnarray*}

from which we deduce that if $s\leq 2$,
\begin{eqnarray*}
\|\delta u^{int,0}_K(t)\|_{H^s} &\leq & C K^{1/2}\sqrt{\eps
\nu}\sum_j\left\{\|\delta^j_3(0)\|_4 +\|\delta^j_3(t)\|_4
\right\}\\
&+&CK^{1/2}\sqrt{\eps \nu}\sum_j \left\{\int_0^t \|\d_s\delta^j_3(s)
\|_4\:ds+ \sup_{s\in[0,t]} \|\delta^j_3(s) \|_4\right\}   \\
&+& \|\delta u^{int,0}_{K|t=0}\|_{H^s}.
\end{eqnarray*}
We now choose $K$ such that
$$
\sqrt{\frac{\nu}{\eps}}\frac{1}{K^{3/2}}= K^{1/2}\sqrt{\eps \nu},
$$
i.e. $K= \eps^{-1/2}.$ We infer that the error term in the evolution
equation is of order $\eps^{1/4}\nu^{1/2}$ in $L^\infty([0,T),
L^2(\om)),$ and that
\begin{eqnarray*}
\|\delta u^{int,0}_K\|_{L^\infty([0,T], H^2(\om))}&\leq& C
\eps^{1/4}\nu^{1/2}\sum_j\sup_{t\in [0,T]}\|\delta^j_3(t)\|_4 \\&+& C
\eps^{1/4}\nu^{1/2}\sum_j\int_0^T \|\d_s\delta^j_3(s) \|_4\:ds\\
&+&  \|\delta u^{int,0}_{K|t=0}\|_{H^2(\om)}.
\end{eqnarray*}
The operator $\mathcal U$ is thus defined by
$$
\mathcal U(\gamma; \delta^0,3, \delta^1_3)(t)=\exp\left(
-\frac{t}{\eps} L\right)u^{int}_L(t) + v^{int,0} +\delta u^{int,0}_K ,
$$
where $\bar u^{int} , v^{int,0} ,\delta u^{int,0}_K$ are defined by
\eqref{def:buint}, \eqref{def:vint} and  \eqref{def:delta_uint}
respectively.

\subsection{Choice of the boundary conditions for $u^{BL}$ and $u^{int}$}

We now explain how the boundary conditions are chosen. As before, we
work with $k_h$ fixed. Also, since the boundary conditions are all
almost-periodic with respect to the fast time variable $t/\eps$, we
work with a fixed frequency $\mu\in\bR$. Note that this decomposition
is allowed by the linearity of the equation.

We set
$$
u^{BL}= \mathcal B (\delta^0_h, \delta^1_h),
$$
where the boundary conditions $\delta^0_h, \delta^1_h$ are yet to be defined.

In order to match the boundary conditions \eqref{top}-\eqref{bottom}
with $\sigma=0$, we must take $u^{BL}$ and $u^{int}$ such that
\begin{eqnarray*}
&& \left( u^{BL}_h + u^{int}_h\right)_{|z=0}=o(\delta),\\
&&\d_z\left( u^{BL}_h + u^{int}_h\right)_{|z=1}=o(\delta),\\
&&\left( u^{BL}_3 + u^{int}_3\right)_{|z=0}=o(\sqrt{\eps \nu}\delta),\\
&&\left( u^{BL}_3 + u^{int}_3\right)_{|z=1}= o(\sqrt{\eps \nu}\delta),
\end{eqnarray*}
denoting by $\delta$ the order of magnitude of $\delta^0, \delta^1$,
in a sense to be made clear later on.

We now examine each of the boundary conditions independently.

\begin{itemize}
 \item At $z=0$, the horizontal boundary condition yields
\be
\hat\delta^0_h(\mu,k_h,t) + \mathbf1_{\mu=-\lambda_k} c_k(t)
\begin{pmatrix}  n_1(k)\\n_2(k)   \end{pmatrix}
= 0,\label{def:d0h}
\ee
where the vector $n(k)$ is defined in Appendix A (see
\eqref{base1},\eqref{base2}).
Since $k_h$ is fixed, note that for all $\mu\in\bR$, there exists at
most one $k_3\in\bZ$ such that $\lambda_{k_h,k_3}=-\mu$, and thus the
expression above is well-defined.

\item Let us now tackle the vertical boundary condition at $z=0$.
According to the third section, the vertical component of $u^{BL}$ at
$z=0$ depends on $\delta^0_h$. Precisely, we recall that
$$
\hat u^{BL}_3(\mu,k_h)_{|z=0}=\sqrt{\eps \nu}\sum_{\sigma\in\{-1,1\}}
\frac{\alpha_\sigma^0}{\lambda^\sigma}(ik_1w_{\lambda^\sigma,1} +
ik_2w_{\lambda^\sigma,2}) ,
$$
(up to exponentially small terms), and
$$
(\alpha_-^0, \alpha_+^0)= P^{-1}\hat \delta^0_h(\mu,k_h).
$$
As a consequence, in order that the vertical boundary condition at
$z=0$ is approximately satisfied, we choose
\be
\hat\delta^0_3=  -
\sum_{\sigma\in\{-1,1\}}\frac{\alpha_\sigma^0}{\lambda^\sigma}(ik_1w_{\lambda^\sigma,1}
+ ik_2w_{\lambda^\sigma,2}).\label{def:d03}\ee

\item At $z=1$, $\d_z u^{int}_h$ is identically zero by construction
of the operator $\mathcal U$, and thus we infer $\delta^1_h=0$.
\item Concerning the vertical component at $z=1$, the calculation is
the same as before. Since $\delta^1_h=0$, we deduce that
$\delta^1_3=0$.

\end{itemize}
The above relations \eqref{def:d0h}-\eqref{def:d03} allow us to write
$\delta^0$ in terms of  $u^{int}_L$. Conversely, the equation
\eqref{eq:u_0} on $u^{int}_L$ depends on $\delta^0_3$, and thus on
$\delta^0_h$  through the operator $\mathcal B$. In other words, there
is a coupling between the boundary condition at the bottom for
$u^{BL}$, and the equation satisfied by $u^{int}_L$. Since $u^{int}_L$
is the only non-vanishing term in $L^2$ norm, we choose (as is usually
done in the rotating fluids literature) to write an explicit equation
for $u^{int}_L$, and to express  $u^{BL}$ in terms of $u^{int}_L$.

\subsection{Derivation of the equation for $u^{int}_L$}

We now compute the Ekman pumping term, that is, the right-hand side in
the equation satisfied by $c_k$ (see \eqref{eq:u_0}). Notice that if
$k\in\bZ^3$ and $k_h\neq 0$, then $|\lambda_k|\neq 1$. In other words,
the source term in \eqref{eq:u_0} involves only the part $\bar u^{BL}$
of the boundary layer; precisely, with the notations of section
\ref{boundary-decomposition}, the decay rate of
$u^{BL}(t,\lambda_k,k_h)$ is
$$
\left(\lambda_k^\pm\right)^2=i \left(- \lambda_k\mp 1 \right) + o(1),
$$
which yields (remember that $\Re(\lambda_k^\pm)>0$)
$$
\lambda_k^\pm=\sqrt{1\pm \lambda_k}\exp\left( \mp i \frac{\pi}{4}
\right) + o(1).
$$
Moreover,
\begin{eqnarray*}
(\alpha_-^0,\alpha_+^0)&=&P^{-1}\left[ - c_k(t) ( n_1(k),n_2(k) )
 \right]\\
&=& - c_k(t) (n_-(k),n_+(k)),
\end{eqnarray*}
where
\begin{eqnarray*}
(n_-(k),n_+(k))&:=&P^{-1}( n_1(k),n_2(k) )\\
&=&\frac{1}{2}(n_1(k)+ in_2(k),n_1(k) - in_2(k) ) + o(1).
\end{eqnarray*}

Replacing these expressions in the formula giving $\delta^0_3$, we infer
$$\hat\delta^0_3(-\lambda_k,k_h,t)=c_k(t)\sum_{\sigma\in\{-1,1\}}\frac{n_\sigma(k)}{\lambda_k^\sigma}(ik_1w_{\lambda^\sigma,1}
+ ik_2w_{\lambda^\sigma,2}).
$$
We deduce that $c_k$ satisfies a linear evolution equation with a
damping term, namely
\be\label{eq:c_k}
\frac{d c_k}{dt} + |k_h|^2 c_k + \nu' |k_3|^2
c_k+\sqrt{\frac{\nu}{\eps}}A_kc_k(t)=0,
\ee
where $\nu'=\pi^2 \nu$ and
$$
A_k:=\mathbf1_{k_h\neq 0}\frac{|k_h|}{2\pi
|k|^2}\sum_{\sigma\in\{-1,1\}}\frac{n_\sigma(k)}{\lambda_k^\sigma}(ik_1w_{\lambda^\sigma,1}
+ ik_2w_{\lambda^\sigma,2}).
$$

An estimate of $\Re (A_k)$, where $\Re(x)$ denotes the real part of a
complex number $x$, is computed in Remark \ref{rem:estAk} below.
Using Duhamel's formula, we deduce that
\be
 |c_k(t)|\leq \exp\left( -t\left( |k_h|^2 +\nu'
|k_3|^2+\sqrt{\frac{\nu}{\eps}} \Re(A_k) \right)\right) |\la
N_k,\gamma\ra| .\label{in:ck}\ee
We deduce the following Lemma:
\begin{Lem}
Assume that $\gamma\in V_0$.
Then there exists a unique solution $\bar u^{int}_L\in
L^\infty_{\text{loc}}(\bR_+,V_0)\cap L^2_{\text{loc}}(\bR_+,
H^1_h(\omega))$ of the  equation
\be
\left\{\begin{array}{l}
\ds \d_t u^{int}_L -\Delta_h  u^{int}_L+
\sqrt{\frac{\nu}{\eps}}S\left[  u^{int}_L \right]= 0,\\
 u^{int}_{L|t=0}= \gamma,
\end{array}
 \right.\label{eq:envEkman0}
\ee where the operator $S$ is defined by
\be\label{Ekman1}
 S\left[  u^{int}_L \right]=\sum_{k\in \bZ^3}A_k \la N_k, u^{int}_L\ra N_k.
\ee

\label{lem:envelope}
\end{Lem}
Hence, in the rest of the section, we take
\be
c_k(t)=\hat \gamma_k \exp\left( -\left(|k_h|^2+
\sqrt{\frac{\nu}{\eps}}A_k\right)t \right).\label{def:ck}
\ee
By doing so, we have neglected the vertical viscosity term $\nu\d_z^2$.

\begin{Rmk}{\bf (i)} Notice that with the scaling we have chosen for
the wind-stress, there is no Ekman pumping due to the wind. Indeed,
the Ekman pumping term is of order $\nu\beta$, which vanishes as
$\eps,\nu\to 0$ according to hypothesis \eqref{hyp:scaling}.

{\bf (ii)} We emphasize that the operator $ S$ constructed above
depends on $\nu$ and $\eps$ through the matrix $P$, the vectors
$w_{\lambda^\pm}$ and the eigenvalues $\lambda_k^\pm$. However, it is
useful, for later purposes, to compute the leading order terms in
$A_k$, which amounts to deriving an equation for the limit of the term
$ u^{int}_L$ as $\eps,\nu$ vanish. Hence we now compute the limit of
$A_k$ as $\eps,\nu\to 0$.

Recall that $n_1(k)$ and $n_2(k)$ are given by \eqref{basis}. Thus, at
first order,
\begin{eqnarray*}
A_k&=&\frac{|k_h|}{2\pi
|k|^2}\sum_{\sigma\in\{-1,1\}}\frac{n_1(k)-i\sigma
n_2(k)}{2\lambda_k^\sigma}(ik_1-\sigma
k_2)\\&=&\frac{|k_h|^2}{8\sqrt{2}\pi^2|k|^2}\left[
\frac{1-\lambda_k}{\sqrt{1 + \lambda_k}}(1-i) +
\frac{1+\lambda_k}{\sqrt{1 - \lambda_k}}(1+i) \right] + o(1)\\
&=&R_k + i I_k + o(1)\end{eqnarray*}
where $R_k$ and $I_k$ are real numbers given by
\begin{eqnarray}
R_k&:=&\frac{1-\lambda_k^2}{8\sqrt{2}\pi^2}\left( \frac{1 +
\lambda_k}{\sqrt{1-\lambda_k}}+\frac{1 -
\lambda_k}{\sqrt{1+\lambda_k}} \right)>0\label{def:Rk}\\
I_k&:=&\frac{1-\lambda_k^2}{8\sqrt{2}\pi^2}\left(\frac{1 +
\lambda_k}{\sqrt{1-\lambda_k}} - \frac{1 -
\lambda_k}{\sqrt{1+\lambda_k}}\right).\label{def:Ik}
\end{eqnarray}
The {\bf Ekman operator} appearing in equation \eqref{envelope} is
thus given by the following formula, for $u\in V_0$
\be\label{def:Ekman_op}
S_{Ekman}\left[ u\right]:=\sum_{k\in \bZ^3}(R_k + i I_k) \la N_k,u\ra N_k.
\ee

{\bf (iii)} Recalling the definition of $\lambda_k$, we deduce that
$$
R_k\geq C \frac{|k_h|}{|k|},
$$
and thus for every $k$, for $\eps,\nu$ small enough, we have
$$\begin{aligned}
   \Re (A_k)\geq C \frac{|k_h|}{|k|},\\
\left|\Im (A_k)\right| \leq C.
  \end{aligned}
$$

\label{rem:estAk}
\end{Rmk}

To conclude this paragraph, we now give estimates on the boundary
conditions $\delta^0_h, \delta^0_3$ in the norm $\|\cdot \|_{s}$.

\begin{Lem}
 Assume that $\delta^0_h, \delta^0_3$ are given by
\eqref{def:d0h}-\eqref{def:d03}. Then the following estimates hold
$$
\|\delta^0_h(t)\|_{s}\leq \left(\sum_{k}|k|^{2(s+1)} |c_k(t)|^2
\right)^{1/2}\leq C \| \gamma\|_{H^{s+1}},
$$
and
$$
\|\delta^0_3(t)\|_{s}\leq  \|\delta^0_h(t)\|_{s+1} \leq
C\|\gamma\|_{H^{s+2}(\om)}.$$

\label{lem:est_delta}
\end{Lem}

\begin{proof}
The bound on $\delta^0_h$ is easily deduced from inequality
\eqref{in:ck} together with formula \eqref{def:d0h} and the
Cauchy-Schwarz inequality. Concerning the other bound, let us recall
that if $\mu=-\lambda_k$,  for $k\in \bZ^3$, then the decay rates
$\lambda^\pm(-\lambda_k,k_h)$ satisfy
$$
|\lambda^\pm|\leq C \left(\frac{|k_3|}{|k_h|} + 1\right).
$$
Plugging all this estimate into \eqref{def:d03}  yields the desired inequality.
\end{proof}

\subsection{Estimates on the boundary layer  and corrector terms}
\label{par:corrector-D}

Now that $u^{int}_L$ is rigorously defined by Lemma
\ref{lem:envelope}, we may define the other terms $v^{int,0}$, $\delta
u^{int,0}$ and $u^{BL,0}$. We have gathered in this paragraph some
estimates which are needed in the proof of Theorem \ref{thm}.

$\bullet$ The boundary layer term of order zero, denoted  by $u^{BL}$,
is defined by
$$
u^{BL,0}= \cB(\delta^0_h,0),
$$
where $\delta^0_h$ is given by \eqref{def:d0h}. Thus we deduce that
the decay rates $\lambda^\pm(\mu,k_h)$ in the non-resonant part of the
boundary layer term $u^{BL,0}$ are all of order one. Consequently,
according to \eqref{B-cont}, the boundary layer term $\bar u^{BL,0} $
satisfies
\begin{eqnarray}
\| \bar u^{BL,0}_h (t)\|_{L^2(\omega)}+(\eps\nu)^{-1/2} \| \bar
u^{BL,0}_3 (t)\|_{L^2(\omega)}\nonumber&\leq &\label{est:ubl1}C \|
\delta^0_h(t)\|_{H^1(\omega_h)} (\eps \nu)^{1/4}\\
&\leq & C \|\gamma\|_{H^{1}(\om)}(\eps \nu)^{1/4}.
\end{eqnarray}
Moreover, the definition of $c_k$ entails that $\bar u^{BL,0}$ is an
approximate solution of \eqref{NS1}, with an error term (due to the
fact that $\d_t c_k$ does not vanish) bounded in $L^2([0,T]\times
\omega)$ by
\begin{eqnarray*}
&&(\eps\nu)^{1/4}\left(\int_0^T\sum_k |\d_t c_k(t)|^2\:dt\right)^{1/2}\\
&\leq & 2(\eps\nu )^{1/4}\left(\int_0^T\sum_k \left(|k_h|^2  +
\frac{\nu}{\eps}|A_k|^2 \right)|\hat \gamma_k|^2 e^{ -2t\left(|k_h|^2
+ \sqrt{\frac{\nu}{\eps}}\Re(A_k)\right) }\:dt\right)^{1/2}\\
&\leq & C (\eps\nu)^{1/4}\left( 1 + \sqrt{\frac{\nu}{\eps}}
\right)^{1/2}\left(\sum_{k}(1 + |k|)^3 |\hat \gamma_k|^2\right)^{1/2}.
\end{eqnarray*}
The right-hand side of the above inequality vanishes as $\eps,\nu\to
0$, and thus the error term satisfies the assumption of Proposition
\ref{energy1}.

Notice that the Dirichlet boundary condition at $z=0$ also generates a
resonant boundary layer term, namely
$$
u^{BL,0}=-\frac{1}{2\pi}\sum_{\mu\in\{-1,1\}}\sum_{k_3\in
2\bZ+1}\sum_{\substack{l_3\in \bZ,\\\sgn(l_3)\mu=1}} c_{(0,0,l_3)}(t)
e^{i\mu \frac{t}{\eps}-\nu k_3^2
t}\frac{(-1)^{\frac{k_3-1}{2}}}{k_3^2}M_{k_3}^\mu.
$$
We have clearly
$$
\|u^{BL,0} \|_{L^\infty([0,\infty), L^2(\om))}\leq C \|\gamma \|_{L^2(\om)}.
$$

$\bullet$ The term $v^{int,0}$ is given by \eqref{def:vint}, in which
$\delta^1_3=0$  and $\delta^0_3$ is defined in \eqref{def:d03}. As a
consequence, $v^{int,0}$ satisfies the estimate
\begin{eqnarray}
||v^{int,0}(t)||_{L^2(\omega)}&\leq& C\|
\delta^0_3(t)\|_{L^2(\omega_h)} (\eps \nu)^{1/2} \nonumber\\
&\leq & C \|\gamma \|_{H^2(\om)}(\eps \nu)^{1/2} .\label{est:vint0}
\end{eqnarray}

$\bullet$ At last, the term $\delta u^{int,0}$ is given by equation
\eqref{eq:delta_c_l}. As stated earlier, we choose a special solution
of \eqref{eq:delta_c_l} in order to keep track of the exponential
decay of $\delta^0_3$. Indeed, we have, for all
$k\in\bZ^3\setminus\{0\},$
$$
\hat \delta^0_3(-\lambda_k,k_h,t)=i\hat \gamma_k\exp\left(
-\left(|k_h|^2+\sqrt{\frac{\nu}{\eps}} A_k\right)t
\right)\sum_{\sigma\in\{-1,1\}}
\frac{n_\sigma(k)}{\lambda_k^\sigma}k_h\cdot w^{\lambda_k^\sigma}.
$$
Thus we choose for $\delta c_l$, $|l|\leq K$, the special solution
constructed in Remark \ref{rmk:spec_sol} in  Appendix C.
With this choice, we obtain
\begin{multline*}
 \| \delta u^{int,0}_A(t)\|_{H^2} \leq  \\C\eps^{1/4}\nu^{1/2}\left(
\sum_{k\in\bZ^3} \frac{(1+ |k_3|)^4}{\left|
\frac{1}{|k_3|^3}-\sqrt{\eps\nu}|\Im(A_k)| \right|^2} |\gamma_k|^2
\exp\left( -2 \sqrt{\frac{\nu}{\eps}} \Re(A_k)t\right) \right)^{1/2}.
\end{multline*}

Moreover, we recall (see Remark \ref{rem:estAk}) that there exists a
constant $C$ such that $|\Im(A_k)|\leq C$ for all $k$; and in the
sequel, we will choose $\gamma$ so that $\hat \gamma_k=0$ for $k_3$
large enough. In this case, we have
$$
\frac{1}{|k_3|^3}-\sqrt{\eps\nu}|\Im(A_k)|\geq \frac{1}{2|k_3|^3}
$$
for $\eps,\nu$ small enough and for all $k$ such that $\hat
\gamma_k\neq0$. The above estimate then becomes
\be
\| \delta u^{int,0}_A(t)\|_{H^2} \leq  C\eps^{1/4}\nu^{1/2}\left(
\sum_{k\in\bZ^3} (1+ |k_3|)^{10} |\gamma_k|^2 \exp\left( -2
\sqrt{\frac{\nu}{\eps}} \Re(A_k)t\right)
\right)^{1/2}\!\!.\label{est:delta_uint}
\ee

\subsection{Conclusion: proof of Theorem \ref{thm} when $\sigma=0$}

\label{par:concl_dirichlet}
The idea is to use the construction of the previous paragraphs in
order to compute an approximate solution of the evolution equation
\eqref{NS1}, which satisfies the boundary conditions up to
sufficiently small error terms. We now have to quantify the order of
approximation required on the boundary condition. This is done in
Lemma \ref{stop} in the Appendix, and thus we build interior and
boundary layer terms until the conditions of the Lemma \ref{stop} are
met.

Let us emphasize that equation \eqref{NS1} supplemented with
homogeneous boundary conditions at $z=0$ and $z=1$ is a contraction in
$L^2$. As a consequence, it is sufficient to prove the Theorem for
arbitrarily smooth initial data. Thus, without any loss of generality,
we assume from now on that the initial data $\gamma$ only has a finite
number of Fourier modes, that is
$$
\gamma= \sum_{|k_h|\leq N}\sum_{|k_3|\leq N'} \hat \gamma_k N_k.
$$

Let us now explain the construction in detail.

\noindent$\bullet$ First, we set
$$
u^0:=u^{int} + u^{BL,0},
$$
where $u^{int}$ and $u^{BL,0}$ have been defined in the previous
paragraphs. We have seen that $u^0$ is an approximate solution of the
evolution equation \eqref{NS1}, with error terms which are all $o(1)$
in $L^2$.
 We now evaluate the error on the boundary conditions. Indeed, setting
$\delta u:=u-u^0$,
we have proved that $\tilde u$ is an approximate solution of
\eqref{NS1}, with some boundary conditions $\eta^0$, $\eta^1$, namely
$$
\begin{array}{ll}
\delta u_{h|z=0}= \eta^0_h,\quad& \d_z\delta u_{h|z=1}=\eta^1_h,\\
\delta u_{3|z=0}= \eta^0_3,\quad& \delta u_{3|z=1}= \eta^1_3.
\end{array}
$$
Thus we have to estimate $\delta\gamma:=\delta u_{|t=0}$, together
with the terms $\eta^0$, $\eta^1$.

First, since $ u^{int}_{L|t=0}=\gamma$ and $u^{BL,0}_{res|t=0}=0$, we obtain
\be
\delta\gamma=-\bar u^{BL,0}_{|t=0}- v^{int,0}_{|t=0}- \delta u^{int,0}_{|t=0},
\ee
where $u^{BL,0}_{|t=0}$, $v^{int,0}_{|t=0}$ and $\delta
u^{int,0}_{|t=0}$   satisfy the estimates \eqref{est:ubl1},
\eqref{est:vint0}, and \eqref{est:delta_uint} respectively. Thus
$$
\| \delta \gamma \|_{L^2}\leq C \left(\|\gamma\|_{H^1} (\eps
\nu)^{1/4}+ \|\gamma\|_{H^2} (\eps \nu)^{1/2} +
\|\gamma\|_{H^5}\eps^{1/4}\nu^{1/2}\right) .
$$

Then, by construction of the operators $\mathcal U$ and $\mathcal B$,
the horizontal remainder boundary term at $z=1$ is exponentially
small: indeed, we have $\d_z u^{int}_{h|z=1}=0$, and consequently,
\be
\eta^1_h=-\sum_{\mu,k_h}\sum_{\sigma\in\{-1,1\}}\alpha^0_\sigma\frac{\lambda^\sigma}{\sqrt{\eps\nu}}e^{-\frac{\lambda^\sigma}{\sqrt{\eps\nu}}}w_{\lambda^\sigma}
e^{ik_h\cdot x_h} e^{i\mu\frac{t}{ \eps}}.
\ee
We infer that
\begin{eqnarray}\label{est:d1h}
\|\eta^1_h \|_{0}^2 &\leq & C \exp\left( -\frac{C}{N' \sqrt{\eps\nu}}
\right)\sum_{|k_h|\leq N} \sum_{|k_3|\leq N'} |\hat
\delta^0_h(-\lambda_k,k_h,t)|^2\\\nonumber
&\leq & C \|\delta^0_h \|_0^2\exp\left( -\frac{C}{N' \sqrt{\eps\nu}} \right) .
\end{eqnarray}

Similarly,
\be
\eta^1_3=\sum_{\mu,k_h}\sum_{\sigma\in\{-1,1\}}\alpha^0_\sigma\frac{\sqrt{\eps\nu}}{\lambda^\sigma}i
k_h\cdot w_{\lambda^\sigma}e^{-\frac{\lambda^\sigma}{\sqrt{\eps\nu}}}e^{ik_h\cdot
x_h} e^{i\mu\frac{t}{ \eps}},
\ee
and thus
\be\label{est:d13}
 \|\eta^1_3 \|_0\leq C N'\sqrt{\eps\nu} \exp\left(
-\frac{C}{N'\sqrt{\eps\nu}} \right)\|\delta^0_h\|_{2}.
\ee

The treatment of the vertical boundary condition at $z=0$ is easier.
Indeed, since $\delta^1=0$, we have $\eta^0_3=0$, because
\be \label{est:d03}
\eta^0_3=-\sum_{\mu,k_h}\sum_{\sigma\in\{-1,1\}}\alpha^1_\sigma\frac{\eps\nu}{(\lambda^\sigma)^2}i
k_h\cdot w_{\lambda^\sigma}e^{-\frac{\lambda^\sigma}{\sqrt{\eps\nu}}}e^{ik_h\cdot
x_h} e^{i\mu\frac{t}{ \eps}}=0.
\ee
There remains to compute $\eta_h^0$; because of the terms $\delta
u_K^{int,0}$ and $v^{int,0}$, $\eta_h^0$ is the largest term of all.
Precisely, we have
\begin{eqnarray}
\eta^0_h(t)&=&-\left[ v^{int,0}_{h|z=0}(t) + \delta
u^{int,0}_{K,h|z=0}(t) \right]\\
\nonumber&=&-\sqrt{\eps\nu}\sum_{\mu, k_h\neq 0}ik_h\cdot \frac{\hat
\delta^0_3(\mu,k_h,t)}{|k_h|^2}e^{i k_h\cdot x_h} e^{i\mu
\frac{t}{\eps}}\\
&&- \sum_{k_h}\sum_{|k_3|\leq K}\delta c_k(t)e^{i k_h\cdot x_h}
e^{-i\lambda_k\frac{t}{\eps}}n_h(k),
\end{eqnarray}
and thus there exists a constant $c>0$ such that for all $t\geq 0$
\begin{eqnarray*}
 \| \eta^0_h (t)\|_{L^2}&\leq& C\left(\sqrt{\eps\nu} \| \delta^0_3(t)
\|_{0}+ \| \delta u_K^{int,0}(t)\|_{H^1}\right)\\
&\leq & C\eps^{1/4}\nu^{1/2} \|\gamma\|_{H^6}\exp\left(
-c\sqrt{\frac{\nu}{\eps}}t \right).
\end{eqnarray*}

Now, the remaining boundary terms $\eta^1_h,\eta^1_3, \eta^0_3$ are
all of order $o(\eps)$ according to \eqref{est:d1h}-\eqref{est:d03}.
Notice furthermore that by construction,
$$
\int_{\omega_h} \eta^j_3=0\quad \text{for }j=0,1.
$$
Consequently, $\eta^1_h,\eta^1_3, \eta^0_3$ all match the conditions
of the stopping Lemma \ref{stop}.

\noindent$\bullet$ We now have to continue the construction with the
``bad'' part of the remaining boundary conditions, i.e. $\eta^0_h$.
Let us define the boundary layer term
$$
\delta u^{BL,0}:=\mathcal B(\eta^0_h, 0).
$$
According to \eqref{B-cont},
$$
\| \delta u^{BL,0}\|_{L^\infty([0,\infty),L^2(\om))}\leq C
(\eps\nu)^{1/4}\| \eta^0_h\|_0 \leq C \eps^{1/2}\nu^{3/4} \|
\gamma\|_{H^6},
$$
and $\delta u^{BL,0}$ is an approximate solution of equation
\eqref{NS1} with a $o(1)$ error term. Moreover, notice that for all
$t\geq 0$, for all $s\geq 0$,
$$
\| \delta u^{BL,0}_{3|z=0}(t)\|_{H^s(\omega_h)}\leq C \eps^{3/4}\nu \|
\gamma\|_{H^7}\exp\left( -c\sqrt{\frac{\nu}{\eps}}t \right) .
$$
We deduce that for all $T>0$, for all $s\geq 0$
$$
\| \delta u^{BL,0}_{3|z=0}\|_{L^2((0,T), H^s(\om_h))}\leq C
\eps^{3/4}\nu \| \gamma\|_{H^7}\left( \frac{\eps}{\nu}
\right)^{1/4}=o(\eps).$$
Thus $\delta u^{BL,0}_{3|z=0}$ satisfies the hypotheses of Lemma
\ref{stop}. Additionnally, $ \delta u^{BL,0}_{|z=1}$ is exponentially
small, and thus also satisfies the conditions of Lemma \ref{stop}.

\noindent$\bullet$ We now define the approximate solution $u_{app}$ by
$$
u_{app}:= u^{int} + u^{BL,0}  + \delta u^{BL,0} + w,
$$
where $w$ is defined by Lemma \ref{stop} with the remaining boundary
conditions. By construction, $u_{app}$ is an approximate solution of
the evolution equation \eqref{NS1}, with
$$
u_{app|t=0}= u_{|t=0} + o(1),
$$
and $u_{app}$ satisfies homogeneous boundary conditions at $z=0$ and
$z=1$. By a simple energy estimate analogous to that of Proposition
\ref{energy1}, we deduce that
$$
\|u - u_{app} \|_{L^{\infty}((0,T),L^2)}\to 0\quad \forall T>0.
$$
Since all the terms in $u_{app}$ except $ u^{int}_L$ and $u^{sing,0}$
are $o(1)$ in $L^2$ norm, Theorem \ref{thm} is proved.

\begin{Rmk}
The proof of Theorem \ref{thm} for $ \sigma=0$ is valid for all ranges
of $\eps,\nu$ such that $\eps,\nu\to 0$. In particular, we do not
assume that $\nu=O(\eps)$. However, in the case $\nu\gg\eps$, all the
modes such that $k_h\neq 0$ in $ u^{int}_L$ are of order
$\exp(-c\sqrt{\nu/\eps}t)$, and vanish exponentially for all $t>0$.
Thus the effect of the heterogeneous horizontal modes of the initial
data  vanishes outside an initial layer of size $\sqrt{\eps/\nu}$. On
the other hand, the modes corresponding to $k_h=0$ are not damped, and
give rise to resonant boundary layer term $u^{BL,0}_{res}$.
Eventually, for $t\gg\sqrt{\eps/\nu}$, we have
$$
u(t)\approx \sum_{k_3\in\bZ^*}\hat \gamma_{(0,0,k_3)}N_{(0,0,k_3)} +
u^{BL,0}_{res}.
$$

\end{Rmk}

\section{Towards more realistic models}\label{conclusion}

\subsection{Justification of equation \eqref{NS1} for geophysical models}
\label{oceanic}

We now explain how our results may give some insight on  models of wind-driven oceanic circulation, which we recall below. In general, these equations are too difficult to deal with in complete mathematical generality, and thus crude assumptions are necessary in order to focus on some special phenomena. Since our aim in this paper is to describe  particular kinds of boundary layers occurring at the top and the bottom of rotating fluids, we give  in this regard a few elements on the derivation of the system \eqref{NS1} supplemented with \eqref{top}-\eqref{bottom}. We emphasize that this derivation is  rigourous neither physically (since a number of important physical phenomena will be neglected in the process), nor mathematically. 
Our sole purpose is to present some motivations for the study of equation \eqref{NS0}, and more generally to derive mathematical tools wich may be useful in models of physical oceanography.
\medskip

$\bullet$ As a starting point, we recall that the ocean can be considered as an incompressible fluid with variable density $ \rho$; hence, neglecting in a first approximation the temperature and salinity variations, the velocity $u$ of the oceanic currents satisfies the Navier-Stokes equations, with a Coriolis term accounting for the rotation of the Earth
\begin{equation}
\label{NS-ocean}
\begin{aligned}
\d_t \rho + u\cdot \nabla \rho=0,\\
\rho \left[\d_t u+(u\cdot \nabla )u \right]+\nabla p =\cF +  \rho u\wedge\Omega \,,\\
\nabla \cdot u =0\,,
\end{aligned}
\end{equation}
where $\cF$  denotes as in the first section the  frictional
force acting on the fluid, $\Omega$ is the (vertical component of the) Earth rotation vector, and $p$ is the pressure defined as the Lagrange multiplier associated with the incompressibility constraint.

We  assume that  the movement to be studied occurs at 
midlatitudes. At such latitudes,   we can {\bf neglect the variations of the Coriolis parameter} $\Omega$ and use the
 $f$-plane
approximation, which makes the  analysis much simpler 
than in the
case of the full model. 

The observed persistence over several days of large-scale waves  in the oceans shows that {\bf frictional forces} $\cF$ are weak, almost  everywhere, when
compared with the Coriolis acceleration and the pressure gradient, but large when compared with the kinematic viscous dissipation of water. 
One common but not very precise notion is that  small-scale
motions, which appear sporadic or on longer time scales, act to  smooth and mix
properties on the larger scales by processes analogous to molecular,  diffusive
transports.
For the present purposes it is only necessary to note that one way to  estimate
the dissipative influence of smaller-scale motions is to retain the same
representation of the frictional force 
$$\cF= A_h \Delta_h u+A_z \d_{zz} u$$
where $A_z$ and $A_h$ are respectively the vertical and horizontal turbulent viscosities, of much larger magnitude than the
molecular value, supposedly because of the greater efficiency of  momentum
transport by macroscopic chunks of fluid. 
Notice that $A_z\neq A_h$  is therefore natural in geophysical framework (see \cite{pedlovsky}). Moreover, models of oceanic circulation usually assume that the vertical viscosity $A_z$ is not constant (see \cite{BD,PP}); we will come back on this point later on.

\bigskip\noindent
$\bullet$
Let us now describe the boundary conditions associated with \eqref{NS-ocean}: typically, Dirichlet boundary conditions are enforced at the bottom of the ocean and on the lateral boundaries of the horizontal domain $\omega_h$ (the coasts), i.e.
\be
\begin{aligned}
u_{|z=h_B(x_h)}=0 \quad\text{(bottom)},\\
u_{|x\in \d \omega_h}=0\quad\text{(coasts).}
\end{aligned}
\ee
In equation \eqref{NS0}, we have  {\bf neglected the  effects of the lateral boundary conditions} by considering the case when $\omega_h$ is the two-dimensional torus.  Of course such an assumption is not physically relevant. It is well known for instance that the lateral boundary layers, called Munk layers,  play a crucial role in the oceanic circulation, in particular in the western intensification of  currents. Moreover, for the sake of simplicity, we did not take into account  the {\bf topography of the bottom} in \eqref{bottom}. The topographic effects described by the function $h_B$ should actually modify the Ekman boundary layer and consequently the limit equations, even if the variations of the bottom are small  (see \cite{DG} and \cite{dgv} for instance).

We assume that the upper surface, which we denote by $\Gamma_s$, has an equation of the type $z=h_S(t,x_h)$.
As boundary conditions on $\Gamma_s$, we enforce (see \cite{GP})
\begin{equation}
\label{top-bis}\begin{aligned}
\Sigma\cdot n_{\Gamma_s}=\sigma_w,\\
\frac{\d}{\d t}\mathbf 1_{0\leq z \leq  h_S(t,x)} + \mathrm{div}_x(\mathbf 1_{0\leq z \leq  h_S(t,x_h)} u)=0\end{aligned}
\end{equation}
where $\Sigma$ is the total stress tensor of the fluid, and $\sigma_w$ is a given stress tensor describing the {\bf wind on the surface of the ocean}. In general, $\Gamma_s$ is a free surface, and a moving interface between air and water, which has its own self consistent motion. In \eqref{top}, we have assumed that
$$
h_S(t,x_h)\equiv D,
$$
where $D$ is the typical depth of the ocean. Hence  (\ref{top}) is a rigid lid approximation, which is a drastic, but standard simplification.
The justification of (\ref{top}) starting from a free surface is mainly open from a mathematical point of view;  we refer to \cite{AP} for the derivation of Navier-type wall laws for the Laplace equation, under general assumptions on the interface, and to \cite{LT} for some elements of justification in the case of the great lake equations.  Nevertheless, from a physical point of view, the simplification does not appear so dramatic, since in any case the free surface is so turbulent with waves and foam, that only modelization is tractable and meaningful. Condition (\ref{top}) is a simple modelization which already catches most of the physical phenomena (see \cite{pedlovsky}).

\noindent $\bullet$ Let us now evaluate the order of magnitude of the different parameters occuring in \eqref{NS-ocean}, and write the equations in a nondimensionalized form. First, since the variations of density are of order $10^{-3}$, we neglect the effects of the variations of $\rho$ in \eqref{NS-ocean} and we assume that
$$
\rho\equiv \rho_0= 10^3\ \mathrm{kg\cdot m^{-3}}.
$$
Moreover, we set
$$\begin{aligned}
   u_h= U u_h',\quad
u_3=W u_3',\\
x_h=H x_h',\quad
z=Dz',\\
  \end{aligned}
$$
where $U$ (resp. $W$) is the typical value of the horizontal (resp. vertical) velocity, $H$ is the horizontal length scale, and $D$ the depth of the ocean. In order that $u'(x')$ remains  divergence-free, we choose
$$
W= \frac{U D}{H}.$$
Typical values for  the mesoscale eddies that have been observed in western Atlantic (see for instance \cite{pedlovsky}) are
$$
U\sim  1 \ \mathrm{cm\cdot s^{-1}} ,\quad H\sim 100\ \mathrm{km}, \quad\text{and } D\sim 4\  \mathrm{km}.
$$
With these values, we get
$$
\eps:= {U\over H \Omega }\sim  10^{-3} ,
$$
and hence $\eps\ll 1$ (notice that the parameter $\eps$ is dimensionless). Thus the asymptotic of fast rotation (small Rossby number) is valid.

\vskip1mm

A typical value of the horizontal turbulent velocity is $A_h\sim 10^6 \, \mathrm{kg\cdot m^{-1} \cdot s^{-1}}$ (see \cite{BD}), which yields
$$
{A_h\over \rho_0 U H} \sim 1.
$$
In general,  the vertical eddy viscosity $A_z$ is not assumed to be constant; in \cite{BD,PP}, the authors consider a vertical viscosity which takes the form 
$$
A_z=\rho_0\left( \nu_b + \nu_0\left( 1 - \frac{5 g \d_z \rho}{\rho_0|\d_z u_h|^2}  \right)^{-2} \right)
$$
and they assume in their numerical computations that $A_z \geq 1  \ \mathrm{kg \cdot m^{-1} \cdot s^{-1}}$. The quantify 
\be
Ri:=- (g \d_z \rho)/(\rho_0|\d_z u_h|^2)\label{Richardson}
\ee is called the local Richardson number. Equation \eqref{NS1} corresponds to a constant approximation for the viscosity $A_z$; this is largely inaccurate, since according to \cite{PP}, measurements show that the value of $A_z$ is usually large  inside the boundary layer (say, $3$ to $10 \ \mathrm{kg\cdot m^{-1} \cdot s^{-1}}$), but substantially smaller  in the interior (under the thermocline). However, since we are primarily interested in the boundary layer behaviour, we only retain the typical boundary layer value $A_z\sim 5 \ \mathrm{kg\cdot m^{-1} \cdot s^{-1}}$, which yields
$$
\nu:={H A_z\over \rho_0 U D^2}\sim  5\cdot 10^{-3}.
$$
Hence we also have $\nu\ll 1$, which justifies our assumption of vanishing vertical viscosity. Notice that the parameter $\nu$ is also dimensionless.

Thus the nondimensionalized system (see for instance \cite{pedlovsky,gill}) becomes
\be\label{NS}
\begin{aligned}
\d_t u' + u'\cdot \nabla u' + \frac{1}{\eps} e_3\wedge u' + \begin{pmatrix} \nabla_h p' \\ \frac{1}{\delta^2}\d_z p'
                                                        \end{pmatrix}
-\Delta_h u' -\nu \d_{zz} u'=0,\\
\nabla\cdot u'=0,
\end{aligned}
\ee
where $\delta:=D/H$ is the aspect ratio. The boundary conditions are \eqref{top}, \eqref{bottom}, with
$$
\beta:= \frac{|\Sigma| D }{A_z U}.
$$
The equation for the boundary layers at $z=1$ and $z=0$ in the above system is exactly the same as in \eqref{NS1}. Thus, we believe that the phenomena we have highlighted (atypical size of boundary layers for resonant forcing, possible destabilization of the fluid for large times) may prove to be useful when studying models of oceanic circulation. 
However, we do not claim that our results truly apply as such to realistic geophysical models, since, as mentionned above, a series of drastic simplifications have been made. Furthermore, some assumptions of Theorem \ref{thm}, such as \eqref{hyp:scaling}, are purely technical, and do not have any physical ground.
Thus, we now turn  to some possible mathematical extensions of Theorem \ref{thm} to more realistic models.

\subsection{Possible extensions}

The previous study allows to characterize the linear response  of a rotating incompressible fluid to some surface stress, which admits fast time oscillations and
may be resonant with the Coriolis force.  In addition to the usual
Ekman layer, we have exhibited another - much larger -  boundary
layer, and  a resonant boundary layer term, the size of which depends on time.
Note that these effects do not modify the mean motion (i.e. the $L^2$
asymptotics) when considering moderate times, say for instance $t\ll
\frac1\nu$.

$\bullet$ {\bf Extensions to nonlinear equations.}
In order to take into account more physics in our model, the first
point is to understand the nonlinear response of the fluid to the
same surface stress. In other words, we are interested in the
asymptotic behaviour of the full Navier-Stokes-Coriolis equation \eqref{NS}-\eqref{top}-\eqref{bottom} including in particular the nonlinear
contribution of the convection.

In the case of a non-resonant forcing, the asymptotic motion of the
fluid  is obtained by some filtering method~: there is indeed two
time scales, a rapid time scale at which  the fluid oscillates
according to the modes of the linear penalization, and a slow one
which characterizes the nonlinear evolution of the wave enveloppes.
The boundary effects do not play any role in the nonlinear process
since they are localized in the vicinity of the surface. They
contribute to  the envelope equations only by the Ekman pumping.
In the case of a resonant forcing,  the boundary effects - which are
not expected to be localized in the same way -  could play a
different role.

$\bullet$ {\bf Towards more physically relevant models.}
The present theory of the wind-driven circulation of a fluid of
uniform density is  actually inadequate to capture the velocity
structure of the oceans. We indeed expect the wind forcing to modify
in depth the circulation. The profile arising from the resonant part
of the forcing and the Ekman pumping are not enough to get a relevant
description of that vertical structure.

We will mention here many phenomena that have been neglected in our
study and which seem to be crucial to obtain realistic models.

\begin{enumerate}[\bf(i)]
\item we first need to consider the {\bf variations of the Coriolis
parameter}, keeping at least the $\beta$-plane approximation~:
$$\Omega =f+\beta y$$
where $y$ is the coordinate measuring the latitude. Such a spatial
dependence of $\Omega$ is necessary to derive Sverdrup's theory of
horizontal transport, which is still one of the foundations of all
theories of the ocean circulation (see \cite{pelovsky2} for instance).

 From a mathematical point of view, we refer to \cite{DG}\cite{GSR2}
and references therein for some preliminary studies on inhomogeneous
rotating fluids.

\item the vertical structure of the ocean cicrulation is also related
to the variations of the density $\rho$, the so-called {\bf
stratification} of the oceans. The theoretical works of Rhines and
Young \cite{RY} have brought some understanding about geostrophic
contours, potential vorticity homogeneization and their role in
shaping the pattern of circulation.  Luyten, Pedlosky and Stommel
\cite {LPS} have then developped a theory for the full density and
velocity structure of the wind-driven circulation by going beyond the
quasi-geostrophic approximation to consider the important effect of
the ventilation of the thermocline which occurs as oceanic density
surfaces rise to intersect the oceanic mixed layer.

However, to our knowledge, there is no mathematical contribution on
that topics, the first difficulty being to determine some suitable
functional framework to deal with the inhomogeneous incompressible
Navier-Stokes equations. Moreover, the behaviour of the fluid is expected to depend in a crucial way on the order of magnitude of the Richardson number $Ri$, defined in \eqref{Richardson} above: when $Ri $ is small (say, $Ri<1/4$), instabilities may develop, leading in turn to some turbulent mixing across layers of equal density. We refer to \cite{Turner} for more details.

\item  we finally have to take into account the bottom topography
which may have an important contribution to the mean circulation as
proved for instance in \cite{DG} or \cite{dgv}.

\end{enumerate}

The crucial point  to understand these features from a mathematical
point of view is to get a  description of the boundary layer operator
which is not based on the Fourier transform, but on the spectral
decomposition of the Coriolis operator. The Coriolis penalization
becomes indeed in the two first cases a skew-symmetric operator with
non-constant coefficients (depending on $\Omega$ and $\rho$). We
therefore have to develop new tools to obtain the asymptotic
expansions in a more abstract and systematic way.

\section*{Appendix A: spectral results on the Coriolis operator}

For the sake of completeness, we recall here - essentially without proof -  some fundamental properties of the Coriolis operator leading to \eqref{basis}.
For a detailed study of these spectral properties we refer for instance to \cite{CDGG}.

Extending any $u\in V_0 $ on $[-1,1]\times \bT^2$ as follows
\begin{equation}
\label{extension}
u _h (x_h,z) =u _h(x_h,-z) \quad \hbox{ and }u _3 (x_h,z) =-u _3(x_h,-z) 
\end{equation}
(which is compatible with the incompressibility constraint $\nabla\cdot u=0$) we obtain a periodic function, so that it is possible to use  some Fourier decomposition.

Setting
\be\label{base1}\left\{\begin{array}{l}
 \ds n_1(k)=\frac{1}{2\pi|k_h|}(ik_2+k_1\lambda_k)\\
 \ds n_2(k)=\frac{1}{2\pi|k_h|}(-ik_1+k_2\lambda_k)\\
 \ds n_3(k)=i\frac{|k_h|}{2\pi\sqrt{|k_h|^2 + (\pi k_3)^2}}
 \end{array}\right.
\quad\text{if }k_h\neq 0, \ee
and
\be\label{base2}
\left\{\begin{array}{l}
 \ds n_1(k)=\frac{\sgn(k_3)}{2\pi}\\
 \ds n_2(k)=\frac{i}{2\pi}\\
\ds n_3(k)=0 \end{array}\right.\text{ else},
\ee
what can be proved actually  is that the family $(N_k)$ defined by
$$ N_k =\exp (ik_h\cdot x_h)\begin{pmatrix} n_1(k)\cos(\pi k_3 z)\\
 n_2(k)\cos (\pi k_3 z)\\ n_3(k)\sin(\pi k_3 z)\end{pmatrix}$$
 is an hilbertian basis of $V_0$
 constituted of eigenvectors of the linear penalization, satisfying \eqref{basis}.

\section*{Appendix B: the stopping condition}

We have postponed here the statement and the proof of the stopping condition  since it is just a technical result (based on straightforward
computations) which is used in several places (Sections 4 and 5).

\begin{Lemma}[Stopping condition]\label{stop}
Let $\delta^0,\delta^1\in L^\infty(\bR^+, H^3(\omega_h))$ be two families  such that $$\int (\delta^1_3-\delta^0_3) dx_h=0$$
and
$$\frac1\eps \| \delta^i\| _{H^1(\omega_h)} \to 0,\quad  \| \delta^i\| _{H^3(\omega_h)} \to 0\hbox{ and } \|\d _t  \delta^i \| _{H^1(\omega_h)} \to 0\hbox{ as } \eps \to 0$$
Then there exists a family $w \in L^\infty(\bR^+, L^2(\Omega))$ with $\nabla\cdot w=0$ such that
$$
w_{|z=0} =\delta^0, \quad w_{3|z=1} =\delta^1_3 \hbox{ and } \d_z w_{h|z=1} =\delta^1_h
$$
and satisfying the following estimates
$$
\| w \|_{L^2(\Omega)} \to 0\hbox{ and }
\left\| \d_t w + \frac1\eps Lw -\nu\d_{zz} w-\Delta_h w \right\| _{L^2(\Omega)} \to 0\hbox{ as } \eps \to 0.
$$
\end{Lemma}

\begin{proof}
Here we have to build a family $w \in L^\infty(\bR^+, L^2(\Omega))$
with $\nabla\cdot w=0$ such that
$$
w_{|z=0} =\delta^0, \quad w_{3|z=1} =\delta^1_3 \hbox{ and } \d_z w_
{h|z=1} =\delta^1_h.
$$
Of course it is not uniquely defined. We just want to obtain one such
family satisfying further suitable continuity estimates.

Given any two-dimensional vector field $w_h$, we get a divergence-
free vector field by setting
$$w_3(x_h,z) =w_3(x_h,0) -\int_0^z (\d_1 w_1+\d_2 w_2)(x_h,z') dz'.$$
In order that the boundary conditions on $w_3$ are satisfied, the
only condition on $w_h$ is therefore
$$\int_0^1 (\d_1 w_1+\d_2 w_2)(x_h,z') dz' +\delta^1_3(x_h) -\delta^0_3(x_h)=0\,.$$

We therefore choose
$$\begin{aligned}
w_1(x_h,z) =\delta^0_1 (x_h)+ \delta^1_1(x_h) z +\d_1 \phi (x_h)
z(1-z)^2,\\
w_2(x_h,z) =\delta^0_2 (x_h)+ \delta^1_2 (x_h)z +\d_2 \phi (x_h)
z(1-z)^2 ,
\end{aligned}
$$
with
$$
\nabla_h\cdot \delta^0_h +\frac12 \nabla_h\cdot \delta^1_h+
\frac1{12} \Delta_h \phi+\delta^1_3 -\delta^0_3=0\,.$$

\bigskip
Standard elliptic estimates give for any $s\geq 0$
$$\| \phi \|_{H^{s+1}(\omega_h)} \leq  C(\| \delta^0\|_{H^s(\omega_h)}
+\| \delta^1\|_{H^s(\omega_h)}).$$
Therefore
$$ \| w\|_{H^2(\Omega)} \leq  C(\| \delta^0\|_{H^3(\omega_h)}+\|
\delta^1\|_{H^3(\omega_h)})$$
so that, using the assumptions on $\delta^0,\delta^1$,
$$\|w\|_{H^2(\Omega)} \to 0 \hbox{ as }\eps \to 0.$$

\bigskip
Furthermore, since $w$ is given in terms of $\delta^0, \delta^1$ by
linear relations with constant coefficients,
$$ \| \d_t w\|_{L^2(\Omega)} \leq  C(\| \d_t \delta^0\|_{H^1
(\omega_h)}+\| \d_t \delta^1\|_{H^1(\omega_h)}).$$
We conclude, using again the assumptions on $\delta^0,\delta^1$  that
$$
\left\| \d_t w + \frac1\eps Lw -\nu\d_{zz} w-\Delta_h w \right\| _{L^2
(\Omega)} \to 0\hbox{ as } \eps \to 0.
$$
\end{proof}

\section*{Appendix C: the small divisor estimate}

We recall here the by-now standard arguments used to obtain some  estimate for the solution to  fast-oscillating linear equation with non-resonant source terms~:
\begin{equation}
\label{source}
\d_t w+\frac1\eps \bP(w) -\nu \Delta_h w-\nu \d^2_{zz} w= \Sigma
\end{equation}
where the horizontal Fourier mode $l_h$ is fixed and 
$$\Sigma (t)=e^{il_h\cdot x_h}\sum_{\mu} \sum_{k_3\in \bZ\atop \mu\neq -\lambda_k} s(\mu,k,t)e^{i\mu \frac t\eps}N_k. $$
We further assume that the frequencies $\mu$ belong either to $\{-\lambda_l, l_3\in \bZ^3\}$, or to some finite set $M$.

The small divisor estimate is the following:
\begin{Lemma}\label{smalldivisor}
Let $w$ be the solution of (\ref{source}), i.e. for all $l=(l_h,l_3)$ with $l_3\in \bZ$,
$$ \d_t w_l + (|l_h|^2+\nu' |l_3|^2)w_l=\sum_{\mu\neq -\lambda_l}  s(\mu,l,t)e^{i(\mu+\lambda_l) \frac t\eps}$$
where $\nu'=\pi^2 \nu.$
 
Then  there exists a constant $C$ such that for all $t>0$, $r>0$, for all $K>0$, we have
\begin{eqnarray*}
\| \bP_Kw(t) \|_{H^r(\om)}&\leq& C \eps \left\{\| s_{|t=0}\|_{r,K}\exp\left(-(|l_h|^2 + \nu' l_3^2)t\right) +\|s(t)\|_{r,K}     \right\}\\
&+&C\eps\int_0^t \|\d_u s(u) \|_{r,K}\exp\left(-(|l_h|^2 + \nu' l_3^2)(t-u)\right)\:du \\&+&C\eps \sup_{u\in[0,t]} \|s(u) \|_{r,K}\\
&+& \| \bP_Kw_{|t=0} \|_{H^r(\om)},
\end{eqnarray*}
where the norm $\| \cdot \|_{r,K}$ is defined by
\begin{eqnarray*}
\| s(t)\|_{r,K}^2&:=&\sum_{|l|\leq K}\sum_{k_3\in \bZ \atop k_3\neq l_3}|k_3|^8|l|^{2r} |s(-\lambda_k,l,t)|^2\\
&+&\sum_{|l|\leq K}\sum_{\mu\in M\atop \mu\neq -\lambda_l}\left( 1 + \mathbf 1_{|\mu|=1}|l|^4 \right)|l|^{2r}|s(\mu,l,t)|^2.
\end{eqnarray*}

\label{lem:delta_uint}

\end{Lemma}
We recall that the notation $\bP_K$ stands for the projection onto the vector space generated by $N_k$ for $|k|\leq K.$

\begin{proof}For all $K>0$, define
$$
w_K:=\bP_K w=\sum_{|k|\leq K} w_l N_l.
$$
We deduce from Duhamel's formula that
\begin{eqnarray}
\nonumber|w_l(t)| &\leq & |w_l(0)|\exp(-(|l_h|^2 + \nu' l_3^2)t)\\
\label{Duhamel}&+& \left|\int_0^t \sum_{\mu\neq -\lambda_l}  s(\mu,l,u)e^{i(\mu+\lambda_l) \frac s\eps}\exp(-(|l_h|^2 + \nu' l_3^2)(t-u))\:du\right|.
\end{eqnarray}
Integrating by parts, we get
\begin{eqnarray*}
&& \left|\int_0^t s(\mu,l,u)e^{i(\lambda_l+\mu)\frac{u}{\eps}}  \exp(-(|l_h|^2 + \nu' l_3^2)(t-u))\:du\right|\\
&\leq & \frac{\eps}{|\lambda_l+\mu|}   |s(\mu,l,t)|  +\frac{\eps}{|\lambda_l+\mu|}  |s(\mu,l_h,0)|e^{-(|l_h|^2 + \nu' l_3^2)t}  \\
&+& \frac{\eps}{|\lambda_l+\mu|}\int_0^t| (|l_h|^2+\nu' |l_3|^2)|s(\mu,l,u)|\exp(-(|l_h|^2 + \nu' l_3^2)(t-u)) \:du\\
&+& \frac{\eps}{|\lambda_l+\mu|}\int_0^t   |\d_u s(\mu,l,u)|\exp(-(|l_h|^2 + \nu' l_3^2)(t-u)) \:du.
\end{eqnarray*}Plugging this inequality back into \eqref{Duhamel}, we deduce that
\begin{eqnarray*}
|w_l(t)|&\leq & |w_l(0)|\exp(-(|l_h|^2 + \nu' l_3^2)t)\\
&&+ C\eps  \sum_{\mu\neq -\lambda_l} \frac{|s(\mu,l,t)|}{|\lambda_l+\mu|}\\
&&+C \eps \sum_{\mu\neq -\lambda_l} \frac{|s(\mu,l_h,0)|}{|\lambda_l+\mu|}\exp(-(|l_h|^2 + \nu' l_3^2)t)\\
&&+C \eps\int_0^t F_l(u)\exp(-(|l_h|^2 + \nu' l_3^2)(t-u))\:du,
\end{eqnarray*}
where
\begin{eqnarray*}
F_l(u)&:=&  \sum_{\mu\neq -\lambda_l} \frac{1}{|\lambda_l+\mu|} |\d_u s(\mu,l,u)| \\
&+& (|l_h|^2 + \nu' |l_3|^2)\sum_{\mu\neq -\lambda_l} \frac{1}{|\lambda_l+\mu|} |s(\mu,l,u)|.
\end{eqnarray*}

There remains to derive bounds for quantities of the type
$$
\sum_{\mu\neq -\lambda_l}\frac{1}{|\mu + \lambda_l|} |s(\mu,l,u)|.
$$
Remember that either $\mu=-\lambda_k$ for some $k=(l_h,k_3)\in\bZ^3$ with $k_3\neq- l_3$, or $\mu\in M$, where $M$ is a finite set. Thus
\begin{eqnarray*}
&&\left(\sum_{\mu\neq -\lambda_l}\frac{1}{|\mu + \lambda_l|} |s(\mu,l,u)|\right)^2\\
&\leq & 2  \left(\sum_{k_3\neq l_3}\frac{1}{| \lambda_l-\lambda_k|} |s(-\lambda_k,l,u)|\right)^2\\
&&+ 2\left(\sum_{\mu\in M}\frac{1}{|\mu + \lambda_l|} |s(\mu,l,u)|\right)^2\\
&\leq &C \sum_{k_3\neq l_3}|k_3|^2\frac{1}{| \lambda_l-\lambda_k|^2} |s(-\lambda_k,l,u)|^2\\
&&+ C \sum_{\mu\in M}\frac{1}{|\mu + \lambda_l|^2} |s(\mu,l,u)|^2.
\end{eqnarray*}
Notice that the function $l_3\mapsto \lambda_l$ is monotonous for $l_h$ fixed. Hence $|\lambda_l-\lambda_k|$ is minimal for $k_3=l_3\pm 1$. Consequently, is is easily checked that for all $l_3\in\bZ$,
$$
|\lambda_l-\lambda_k|^{-1} \leq C\frac{|k|^3}{|l_h|^2}.
$$
Moreover, if $\mu\in M$, then either $\mu\notin \{0,1,-1\}$, and in this case
$$
|\lambda_l-\mu|^{-1} \leq C,
$$
or $\mu=0$, and then
$$
|\lambda_l-\mu|^{-1} \leq C\frac{|l|}{|l_3|},
$$
or 
$\mu\in \{1,-1\}$, and then
$$
|\lambda_l-\mu|^{-1} \leq C\frac{|l|^ 2}{|l_h|^ 2}.
$$

Gathering all these results we get
\begin{eqnarray}\nonumber
 |w_l(t)|&\leq & |w_l(0)|+C\eps D^0_l(t) \\&&+C\eps\int_0^t D^1_l(u)\exp\left(-\left(\frac{|l_h|^2}{2} + \nu' l_3^2\right)(t-u)\right)\:du,\label{in:delta_cl}
\end{eqnarray}
where
\begin{eqnarray*}
D^0_l(t)&:=&\left[\sum_{k_3}|k_3|^8|s(-\lambda_k,l,0)|^2\right]^{1/2}\exp\left(-(|l_h|^2 + \nu' l_3^2)t\right) \\
&+&\sum_{\mu\in M,\atop\mu\neq -\lambda_l}\left( 1 + \mathbf1_{|\mu|=1}|l|^2 \right) |s(\mu,l,0)|\exp\left(-(|l_h|^2 + \nu' l_3^2)t\right)\\
&+&\left[\sum_{k_3}|k_3|^8|s(-\lambda_k,l,t)|^2\right]^{1/2}\\
&+& \sum_{\mu\in M,\atop\mu\neq -\lambda_l}\left( 1 + \mathbf1_{|\mu|=1}|l|^2 \right)|s(\mu,l,t)|\,
\end{eqnarray*}
and
\begin{eqnarray*}
D^1_l (u)&:=&\left[\sum_{k_3}|k_3|^8|\d_u s(-\lambda_k,l,u)|^2\right]^{1/2}\\
&+&\sum_j \sum_{\mu\in M,\atop\mu\neq -\lambda_l}\left( 1 + \mathbf1_{|\mu|=1}|l|^2 \right) |\d_u s(\mu,l,u)|\\
&+&(|l_h|^2 + \nu'|l_3|^2)\left[\sum_{k_3}|k_3|^8|s(-\lambda_k,l,u)|^2\right]^{1/2}\\
&+&
(|l_h|^2 + \nu'|l_3|^2) \sum_{\mu\in M,\atop\mu\neq -\lambda_l}\left( 1 + \mathbf1_{|\mu|=1}|l|^2 \right)|s(\mu,l_h,u)|.
\end{eqnarray*}The estimate of Lemma \ref{smalldivisor} follows.
\end{proof}

\begin{Rmk}
Assume that the Fourier coefficients of $s$ have exponential decay, meaning that for all $(\mu,l)$, there exists $s_0(\mu,l)\in \bC$, and $c(\mu,l)\in \bC$ with nonnegative real part such that 
$$
s(\mu,l,t)=s_0(\mu,l)\exp(-c(\mu,l)t).
$$
Then provided the sequence $s_0(\mu,l)$ is sufficiently convergent, a special solution of \eqref{source} can be built, which preserves the exponential decay property. Indeed, for all $l\in\bZ^3$, set
$$
w_l(t):=\sum_{\mu\neq -\lambda_l}s_0(\mu,l)\frac{\exp\left( i(\lambda_l + \mu)\frac{t}{\eps} - c(\mu,l) t\right)}{i\frac{\lambda_l+ \mu}\eps - c(\mu,l) + |l_h|^2 + \nu|l_3|^3}.
$$
Then it can be readily checked that $w$ is a solution of \eqref{source}, and moreover
$$
|w_l(t)|\leq \eps\sum_{\mu\neq -\lambda_l}\frac{1}{|\lambda_l+\mu-\eps \Im(c(\mu,l))|}|s_0(\mu,l)|\exp\left( -\Re(c(\mu,l))t \right).
$$

\label{rmk:spec_sol}

\end{Rmk}

\end{document}